\documentclass[a4paper,twoside]{article}
\usepackage{a4}
\usepackage{amssymb}
\usepackage{amsmath}
\usepackage{upref}
\usepackage{xypic}
\usepackage[active]{srcltx}
\usepackage[dvips,colorlinks,citecolor=blue,linkcolor=blue]{hyperref}
\usepackage[dvipsnames]{color}
\allowdisplaybreaks[2] 
\newcount\minutes \newcount\hours
\hours=\time
\divide\hours 60
\minutes=\hours
\multiply\minutes -60
\advance\minutes \time
\newcommand{\klockan}{\the\hours:{\ifnum\minutes<10 0\fi}\the\minutes}
\newcommand{\tid}{\today\ \klockan}
\newcommand{\prtid}{\smash{\raise 10mm \hbox{\LaTeX ed \tid}}}
\renewcommand{\prtid}{}
\makeatletter
\def\sectionmark#1{} 
\def\subsectionmark#1{}
\newcommand{\sectnr}{\ifnum \c@secnumdepth >\z@
                 \thesection.\hskip 1em\relax \fi}
\def\@evenhead{\footnotesize\rm\thepage\hfil\leftmark\hfil\llap{\prtid}}
\def\@oddhead{\footnotesize\rm\rlap{\prtid}\hfil\rightmark\hfil\thepage}
\def\tableofcontents{\section*{Contents} 
 \@starttoc{toc}}
\makeatother
\makeatletter
\def\@biblabel#1{#1.}
\makeatother
\makeatletter
\let\Thebibliography=\thebibliography
\renewcommand{\thebibliography}[1]{\def\@mkboth##1##2{}\Thebibliography{#1}
\addcontentsline{toc}{section}{References}
\frenchspacing 
\setlength{\@topsep}{0pt}
\setlength{\itemsep}{0pt}%
\setlength{\parskip}{0pt plus 2pt}%
}
\makeatother
\makeatletter
\def\mdots@{\mathinner.\nonscript\!.%
 \ifx\next,.\else\ifx\next;.\else\ifx\next..\else
 \nonscript\!\mathinner.\fi\fi\fi}
\let\ldots\mdots@
\let\cdots\mdots@
\let\dotso\mdots@
\let\dotsb\mdots@
\let\dotsm\mdots@
\let\dotsc\mdots@
\def\vdots{\vbox{\baselineskip2.8\p@ \lineskiplimit\z@
    \kern6\p@\hbox{.}\hbox{.}\hbox{.}\kern3\p@}}
\def\ddots{\mathinner{\mkern1mu\raise8.6\p@\vbox{\kern7\p@\hbox{.}}%
    \raise5.8\p@\hbox{.}\raise3\p@\hbox{.}\mkern1mu}}
\makeatother
\makeatletter
\let\Enumerate=\enumerate
\renewcommand{\enumerate}{\Enumerate%
\setlength{\@topsep}{0pt}
\setlength{\itemsep}{0pt}%
\setlength{\parskip}{0pt plus 1pt}%
\renewcommand{\theenumi}{\textup{(\alph{enumi})}}%
\renewcommand{\labelenumi}{\theenumi}%
}
\let\endEnumerate=\endenumerate
\renewcommand{\endenumerate}{\endEnumerate\unskip}
\makeatother
\makeatletter
\def\@seccntformat#1{\csname the#1\endcsname.\quad}
\makeatother
\newcommand{\authortitle}[2]{\author{#1}\title{#2}\markboth{#1}{#2}}
\newcommand{\auth}[2]{{#1, #2.}}
\def\idxauth{\auth}
\newcommand{\art}[6]{{\sc #1, \rm #2, \it #3\/ \bf #4 \rm (#5), \mbox{#6}.}}
\newcommand{\artnopt}[6]{{\sc #1, \rm #2, \it #3\/ \bf #4 \rm (#5), \mbox{#6}}}
\newcommand{\artprep}[3]{{\sc #1, \rm #2, \it #3.}}
\newcommand{\book}[3]{{\sc #1, \it #2, \rm #3.}}
\newcommand{\AND}{{\rm and }}
\RequirePackage{amsthm}
\newtheoremstyle{descriptive}%
  {\topsep}   
  {\topsep}   
  {\rmfamily} 
  {}          
  {\bfseries} 
  {.}         
  { }         
  {}          
\newtheoremstyle{propositional}%
  {\topsep}   
  {\topsep}   
  {\itshape}  
  {}          
  {\bfseries} 
  {.}         
  { }         
  {}          
\theoremstyle{propositional}
\newtheorem{thm}{Theorem}[section]
\newtheorem{prop}[thm]{Proposition}
\newtheorem{lem}[thm]{Lemma}
\newtheorem{cor}[thm]{Corollary}
\theoremstyle{descriptive}
\newtheorem{deff}[thm]{Definition}
\newtheorem{example}[thm]{Example}
\newtheorem{openprob}[thm]{Open problem}
\makeatletter
\renewenvironment{proof}[1][\proofname]{\par
  \pushQED{\qed}%
  \normalfont
  \trivlist
  \item[\hskip\labelsep
        \itshape
    #1\@addpunct{.}]\ignorespaces
}{%
  \popQED\endtrivlist\@endpefalse
}
\makeatother
\newcommand{\setm}{\setminus}
\renewcommand{\emptyset}{\varnothing}
\DeclareMathOperator{\diam}{diam}
\DeclareMathOperator{\dist}{dist}
\DeclareMathOperator{\inner}{inner}
\DeclareMathOperator{\interior}{int}
\newcommand{\bdry}{\partial}
\newcommand{\bdy}{\bdry}
{\catcode`p =12 \catcode`t =12 \gdef\eeaa#1pt{#1}}      
\def\accentadjtext#1{\setbox0\hbox{$#1$}\kern   
                \expandafter\eeaa\the\fontdimen1\textfont1 \ht0 }
\def\accentadjscript#1{\setbox0\hbox{$#1$}\kern 
                \expandafter\eeaa\the\fontdimen1\scriptfont1 \ht0 }
\def\accentadjscriptscript#1{\setbox0\hbox{$#1$}\kern   
                \expandafter\eeaa\the\fontdimen1\scriptscriptfont1 \ht0 }
\def\accentadjtextback#1{\setbox0\hbox{$#1$}\kern       
                -\expandafter\eeaa\the\fontdimen1\textfont1 \ht0 }
\def\accentadjscriptback#1{\setbox0\hbox{$#1$}\kern     
                -\expandafter\eeaa\the\fontdimen1\scriptfont1 \ht0 }
\def\accentadjscriptscriptback#1{\setbox0\hbox{$#1$}\kern 
                -\expandafter\eeaa\the\fontdimen1\scriptscriptfont1 \ht0 }
\def\itoverline#1{{\mathsurround0pt\mathchoice
        {\rlap{$\accentadjtext{\displaystyle #1}
                \accentadjtext{\vrule height1.593pt}
                \overline{\phantom{\displaystyle #1}
                \accentadjtextback{\displaystyle #1}}$}{#1}}
        {\rlap{$\accentadjtext{\textstyle #1}
                \accentadjtext{\vrule height1.593pt}
                \overline{\phantom{\textstyle #1}
                \accentadjtextback{\textstyle #1}}$}{#1}}
        {\rlap{$\accentadjscript{\scriptstyle #1}
                \accentadjscript{\vrule height1.593pt}
                \overline{\phantom{\scriptstyle #1}
                \accentadjscriptback{\scriptstyle #1}}$}{#1}}
        {\rlap{$\accentadjscriptscript{\scriptscriptstyle #1}
                \accentadjscriptscript{\vrule height1.593pt}
                \overline{\phantom{\scriptscriptstyle #1}
                \accentadjscriptscriptback{\scriptscriptstyle #1}}$}{#1}}}}
\newcommand{\al}{\alpha}
\newcommand{\alp}{\alpha}
\newcommand{\ga}{\gamma}
\newcommand{\de}{\delta}
\newcommand{\eps}{\varepsilon}
\newcommand{\Om}{\Omega}
\newcommand{\clOmm}{{\overline{\Om}\mspace{1mu}}^M}
\newcommand{\bdym}{{\bdy_M}}
\newcommand{\p}{{$p\mspace{1mu}$}}
\newcommand{\R}{\mathbf{R}}
\newcommand{\Sphere}{\mathbf{S}}
\newcommand{\Sp}{\mathbf{S}}
\newcommand{\Gt}{\widetilde{G}}
\newcommand{\Et}{\widetilde{E}}
\newcommand{\G}{{\cal G}}
\newcommand{\K}{\mathcal{K}}%
\newcommand{\Kt}{\widetilde{\mathcal{K}}}%
\newcommand{\At}{\tilde{A}}
\newcommand{\Gj}[2]{G_{#1}(#2)}
\newcommand{\Gjbig}[2]{G_{#1}\bigl(#2\bigr)}
\newcommand{\Gjx}[3]{G_{#1}(#2,#3)}
\newcommand{\Gjr}{\Gj{j}{r}}
\newcommand{\Gjrxo}{\Gjx{j}{r}{x_0}}
\newcommand{\Gkr}{\Gj{k}{r}}
\newcommand{\Gkone}{\Gj{k}{1}}
\newcommand{\Gjs}{\Gj{j}{s}}
\newcommand{\Gones}{\Gj{1}{s}}
\newcommand{\Hh}{H}
\newcommand{\Hr}[1]{H{(#1)}}
\newcommand{\Hrx}[2]{H{(#1,#2)}}
\newcommand{\Hs}{\Hr{s}}
\newcommand{\din}{d_{\inner}}
\newcommand{\doverto}{\overset{d\,\,}\to}
\newcommand{\dMto}{\overset{d_M\,}\longrightarrow}
\newcommand{\dM}{d_M}
\makeatletter
\newcommand{\setcurrentlabel}[1]{\def\@currentlabel{#1}}
\makeatother
\newcounter{saveenumi}
\numberwithin{equation}{section}
\newcommand{\eqv}{\mathchoice{\quad \Longleftrightarrow \quad}{\Leftrightarrow}
                {\Leftrightarrow}{\Leftrightarrow}}
\newcommand{\imp}{\mathchoice{\quad \Longrightarrow \quad}{\Rightarrow}
                {\Rightarrow}{\Rightarrow}}
\newcommand{\smalleqv}{\mathchoice{\ \Longleftrightarrow \ }{\Leftrightarrow}
                {\Leftrightarrow}{\Leftrightarrow}}
\newcommand{\smallimp}{\mathchoice{\ \Longrightarrow \ }{\Rightarrow}
                {\Rightarrow}{\Rightarrow}}
\newenvironment{ack}{\medskip{\it Acknowledgement.}}{}

\begin{document}

\authortitle{Anders Bj\"orn, Jana Bj\"orn
    and Nageswari Shanmugalingam}
{The Mazurkiewicz distance and sets that are finitely connected at the boundary}
\author{
Anders Bj\"orn \\
\it\small Department of Mathematics, Link\"opings universitet, \\
\it\small SE-581 83 Link\"oping, Sweden\/{\rm ;}
\it \small anders.bjorn@liu.se
\\
\\
Jana Bj\"orn \\
\it\small Department of Mathematics, Link\"opings universitet, \\
\it\small SE-581 83 Link\"oping, Sweden\/{\rm ;}
\it \small jana.bjorn@liu.se
\\
\\
Nageswari Shanmugalingam
\\
\it \small  Department of Mathematical Sciences, University of Cincinnati, \\
\it \small  P.O.\ Box 210025, Cincinnati, OH 45221-0025, U.S.A.\/{\rm ;}
\it \small  shanmun@uc.edu
}

\date{}
\maketitle

\noindent{\small
{\bf Abstract}.
We study 
local connectedness, local accessibility and finite
connectedness at the boundary,
in relation to the compactness of the Mazurkiewicz completion
of a bounded domain in a metric
space.
For countably connected planar domains we obtain a complete
characterization.
It is also shown exactly which parts of this characterization
fail in higher dimensions and in metric spaces.
}

\bigskip


\medskip


\section{Introduction}

In this paper we study the question of compactness of the completion
of a domain $\Om$ with respect to the \emph{Mazurkiewicz
distance} 
\[
     \dM(x,y) =\inf \diam E,
\]
where the infimum is taken over all connected sets $E \subset \Om$
containing $x,y \in \Om$. The distance is also called \emph{relative distance},
\emph{Mazurkiewicz intrinsic metric}  and \emph{inner diameter
  distance} in the literature.
We characterize the compactness of this completion
using topological properties of the domain
near the boundary, such as finite connectivity at the boundary, local
accessibility and local connectivity of the boundary.
Even in the Euclidean setting these
topological notions are delicate, and some characterizations hold only
in the two-dimensional setting.

This study is 
part of a project which deals with 
extending the notion
of boundary in connection with the Dirichlet problem for partial
differential operators and
variational problems,
see Adamowicz--Bj\"orn--Bj\"orn--Shanmugalingam~\cite{ABBSprime} 
and Bj\"orn--Bj\"orn--Shanmugalingam~\cite{BBSdir}.
The Euclidean
boundary is usually sufficient for smooth enough
domains, but for more general domains a different notion of boundary
is needed.
For example, even in the simple case of a planar slit disc,
 points on the slit can be approached from two different directions, and one should permit the 
 Dirichlet solution to take two different boundary
values corresponding to  these different approaches to that point.
The Mazurkiewicz boundary considered in this paper partially rectifies
this issue, but does not solve the problem completely.
More concretely, in
\cite{BBSdir}
the Perron method for solving the Dirichlet problem for \p-harmonic
functions was developed with respect to the 
Mazurkiewicz boundary, under the assumption that it is compact. 
(For the corresponding problem
with respect to the Euclidean boundary (on $\R^n$) or given metric boundary
(on metric spaces),
see \cite{GranLindMar}, \cite{Kilp89}, \cite{HeKiMa},
\cite{BBS}, \cite{BBS2} and \cite{BBbook}, as well as the references 
in the notes to \cite[Sections~9 and~16]{HeKiMa}.)
Therefore the compactness of the Mazurkiewicz boundary is of
particular interest and importance to us.

In the general setting of proper locally pathconnected
metric spaces we prove the following result.
See Section~\ref{sect-lc} for the definitions of the involved notions.

\begin{thm} \label{thm-clOmm-cpt}
The closure\/ $\clOmm$ with respect to the  Mazurkiewicz distance 
is compact if and only if\/
$\Om$ is finitely connected at the boundary.
\end{thm}

Finite connectivity of $\Om$ at a boundary point $x_0$ is also characterized in terms of components of $\Om\cap B(x_0,r)$, see Proposition~\ref{prop1-fin}.
Proposition~\ref{prop-N-conn} provides a connection between finite
connectivity of $\Om$ at $x_0$ and the number of points in the Mazurkiewicz boundary  
$\bdry_M\Om$ corresponding to $x_0$.

For countably connected planar domains we prove the following
  theorem, which generalizes a result by Newman~\cite{newman}, and follows from the more general Theorem~\ref{thm-newman-R2}.

\begin{thm} 
If\/ $\Om \subset \R^2$ and\/ $\R^2\setm\Om$ has countably many components, then
the following are equivalent\/\textup{:}
\begin{enumerate}
\item 
each  $x_0 \in \bdy \Om$ is locally accessible from\/ $\Om$\/\textup{;}
\item 
for every $x_0 \in \bdy \Om$ we have that\/
$\Om \cup \{x_0\}$ is locally connected at $x_0$\/\textup{;}
\item \label{H-item}
for all $r>0$ and $x_0 \in \bdy\Om$ it is true that $x_0 \notin \itoverline{\Hrx{r}{x_0}}$\textup{;}
\item 
$\Om$ is finitely connected at the boundary\/\textup{;}
\item 
$\clOmm$ is compact.
\end{enumerate}
\end{thm}

For finitely connected planar domains we further show in
Theorem~\ref{thm-newman}  that (a)--(e) are
equivalent to each of
\[ \text{(f) } \bdy \Om \text{ is locally connected} 
\quad \text{and} \quad \text{(g) } 
\R^2 \setm \Om \text{ is locally connected},
\]
while for infinitely connected planar domains both (f) and (g) 
always fail,
see Theorem~\ref{thm-newman-R2}.
For a general bounded domain in a proper locally connected metric
space we only obtain the following implications:
\[
\text{(e)}\  \Longleftrightarrow \ \text{(d)} \ \Longrightarrow \ 
\text{(c)} \ \Longleftrightarrow \ \text{(b)} \ \Longleftrightarrow \ 
\text{(a)} \quad \text{and} \quad
\text{(f)} \ \Longrightarrow \ \text{(g)},
\]
see Theorem~\ref{thm-newman-gen},
We also construct Examples~\ref{ex-R3-finconn-compl-new}--\ref{ex-Jana}
showing that no other implication holds true 
even when $\Om \subset \R^3$ is required to be 
homeomorphic to a ball.

In
  Adamowicz--Bj\"orn--Bj\"orn--Shanmugalingam~\cite[Theorem~9.6]{ABBSprime} 
it was shown that the Mazurkiewicz boundary coincides with the
collection of singleton prime ends of the domain. 
Our results therefore 
also have applications in connection with  prime ends.
In particular, 
in Section~7 of~\cite{ABBSprime}, singleton prime ends
are related to (nonlocal) accessibility of boundary points,
while in Section~10 of~\cite{ABBSprime} compactness of the
  singleton prime end boundary is discussed.
(Beware that the prime ends introduced in~\cite{ABBSprime} do not
always coincide with the classical Carath\'eodory prime ends for
simply connected planar domains.)

As to be expected, the discussion in this paper deals with purely metric 
 topological notions, and hence might be of interest to a wider
 audience. 
Some parts of this paper are related to the results of
Rempe~\cite{rempe}, which characterizes local connectivity at the
boundary in terms of prime ends for simply connected planar domains, and to
R.~L.~Moore~\cite{moore62}. 
However, the reader should beware that the definitions of compactness,
connectedness and continuum in~\cite{moore62} are different from the
modern definitions.

The paper is organized as follows.
In Section~\ref{sect-lc} we describe the topological notions used
throughout the paper, while in Section~\ref{sect-dM} we define and
discuss the Mazurkiewicz distance, and in particular prove
Theorem~\ref{thm-clOmm-cpt}.
In Section~\ref{sect-loc-conn-bdry} we describe relations between the
finite connectivity at the boundary and other topological notions.

In the last section we revisit planar domains and show that a countably
connected planar domain is finitely connected at the boundary
whenever each topological boundary point is locally accessible from
it (Theorem~\ref{thm-newman-count-singleton}), generalizing a result 
by Newman~\cite{newman}.

\begin{ack}
This research was begun while the first two authors visited
the University of Cincinnati during the first half year of 2010,
and continued while the third author visited Link\"opings universitet
in March 2011, and during the stay of the three authors at Institut
Mittag--Leffler in Autumn 2013. 
We wish to thank these institutions for their kind hospitality.
We also wish to thank Tomasz Adamowicz and Harold Bell for fruitful discussions.

The first two authors were supported by the Swedish Research Council.
The first author was also a Fulbright scholar during his
visit to the University of Cincinnati, supported by the Swedish
Fulbright Commission,
while the second author was a Visiting Taft Fellow
during her visit to the University of Cincinnati, 
supported by the Charles Phelps Taft Research Center at the University
of Cincinnati.
The third author was also supported by the Taft Research Center 
of the University of Cincinnati and by 
grant~\#200474 from the Simons Foundation and NSF grant DMS-1200915.
\end{ack}

\section{Local connectedness}  
\label{sect-lc}

\emph{Throughout the paper we let $(X,d)$ be a metric space and let
$\Om \subset X$ be  a bounded domain,
i.e.\  $\Om$ is a bounded nonempty open connected set.}

\medskip

In this section we describe various notions of connectedness which are useful 
in the study of topological properties of the Mazurkiewicz boundary.
We also characterize finite and bounded connectivity of $\Om$ at a boundary 
point $x_0$ in terms of components of $B(x_0,r)\cap\Om$.

There are four distinct types of local connectedness one
can consider at a point in a metric space.

\begin{deff} \label{deff-lc}
A topological space $Y$ is \emph{locally\/ \textup{(}path\/\textup{)}connected} 
at $x\in Y$
if for every $r>0$ there is a (path)connected neighbourhood
$G \subset B(x,r)$ of $x$.

The space $Y$ is \emph{\textup{(}path\/\textup{)}connected im kleinen} at $x\in Y$
if for every $r>0$ there is a (path)connected set
$A \subset B(x,r)$ such that $x \in \interior A$.

We also say that $Y$ has any of these properties \emph{globally} if it holds
for all $x \in Y$.
\end{deff}

Here $B(x,r)=\{y \in X : d(x,y)<r\}$ and
neighbourhoods are open.

The terminology is not standardized throughout the literature.
Indeed, Kuratowski's definition in \cite{kuratowski2} of
local (arcwise) connectedness corresponds to the above definition of
(path)connectedness im kleinen.
Whyburn's definition in \cite{whyburn42} of
local  connectedness is our definition of
connectedness im kleinen. We follow the terminology of
Munkres~\cite{munkres2} in defining
local connectedness. However, \cite{munkres2} uses the 
name weakly locally connected instead of 
connected im kleinen.

It is obvious that we have the following implications:
\begin{equation} \label{eq-lc-trivial}
\begin{split}
\xymatrix{
   \text{ locally pathconnected at $x$ } \ar@2[d] \ar@{=>}[r]
  &   \text{ locally connected at $x$ } \ar@2[d] \\
   \text{ pathconnected im kleinen at $x$} \ar@{=>}[r]
  &   \text{ connected im kleinen  at $x$. }
}
\end{split}
\end{equation}

For the corresponding global statements 
(i.e.\ the assumptions are required at all $x \in X$)
it is known that
also the corresponding upward implications hold.
Indeed, if $X$ is (path)connected im kleinen at every $x \in X$ and $r>0$,
then the (path)connected component of $B(x,r)$ containing $x$ must be open,
see e.g.\ Whyburn~\cite[p.~20, (14.1)]{whyburn42},
and hence $X$ is locally (path)connected at $x$.

If $X$ is moreover complete, then all four global statements are equivalent
by Mazurkiewicz--Moore--Menger's theorem, see
Kuratowski~\cite[p.~154, Theorem~1]{kuratowski2}.
On the other hand, Moore~\cite{moore26}
gave an example of a
locally connected metric space which is not locally pathconnected.
In fact his example is a
 noncomplete subset of $\R^2$ which is connected
and locally connected but contains no path.

We are mostly interested in considering  connectedness properties 
at points on the boundary of $\Omega$.
Fix a boundary point $x_0 \in \bdy \Om$ throughout the rest of this section.

\begin{deff}\label{deff-finite-conn}
We say that $\Om$ is \emph{finitely connected} at $x_0\in \bdy \Om$ if
for every $r>0$ there is an open set $G\subset X$ 
such that $x_0 \in G \subset B(x_0,r)$
and $G \cap \Om$ has only finitely many components.

If there is $N>0$ such that for every $r>0$
there is an open $G\subset X$ 
such that $x_0 \in G \subset B(x_0,r)$
and such that $G \cap \Om$ has at most $N$ components,
then we say that $\Om$ is
\emph{boundedly connected} at $x_0$.
If moreover $N$ is minimal, then $\Om$ is said to be
\emph{$N$-connected} at $x_0$. 
Furthermore,
$\Om$ is \emph{locally connected} at $x_0 \in \bdy \Om$
if it is $1$-connected at $x_0$.

We say that $\Om$ has one of these properties 
\emph{at the boundary} if it has that property  at all boundary points.
\end{deff}

The terminology above follows N\"akki~\cite{nakki70}
who seems to have first used this terminology in print.
(N\"akki~\cite{nakki-private} has informed us that
he learned about the terminology from V\"ais\"al\"a, who
however first seems to have used it in print in~\cite{vaisala}.)
The concept of finite connectedness at the boundary was already
used by Newman~\cite{newman} (only in the first edition), but without a name.

Beware that the notion of finitely connected domains is a 
completely different notion (also used later in this paper). 
Similarly, it is important to distinquish between ``locally connected'', 
``locally connected at a point'' (as in Definition~\ref{deff-lc}), 
``locally connected at a boundary point'' and ``locally connected at the 
boundary'' (as in Definition~\ref{deff-finite-conn}).
For example, note that for $\Om$ to be locally connected at $x_0\in\bdry\Om$
(in the sense of Definition~\ref{deff-finite-conn}) it is necessary that  
$\Om\cup\{x_0\}$ is locally connected at $x_0$ 
(in the sense of Definition~\ref{deff-lc}). The converse
need not hold, as evidenced by the slit disc 
$\Om=\{(x,y)\in\R^2\, :\, x^2+y^2<1\}\setm ([0,1)\times\{0\})$, 
with $x_0=\bigl(\tfrac{1}{2},0\bigr)\in\R^2$. 
See Theorem~\ref{thm-newman-pt} for the relations between these
various notions of connectedness.

Theorems~1.10 and~1.11 of \cite{nakki70}
give several characterizations
of finite connectedness and $N$-connectedness.
We provide some further characterizations in
this section, which 
will be useful for us in the study of 
the Mazurkiewicz boundary.
The results of~\cite{nakki70} are formulated for $\overline{\R^n}$, but the proofs
of these topological properties given there also hold in metric spaces.

\begin{example}
The following was shown in 
Adamowicz--Bj\"orn--Bj\"orn--Shan\-mu\-ga\-lin\-gam~\cite{ABBSprime}.
(See \cite{ABBSprime} for definitions of the concepts mentioned above and below.)
The following are true:
\begin{enumerate}
\item
If $\Om$ is a uniform domain, then 
$\Om$ is locally connected at the boundary. 
\item
If $\Om$ is finitely connected at the boundary
and $X$ is locally connected and proper, then
there is a natural one-to-one correspondence between the points
in the
Mazurkiewicz boundary (see Definition~\ref{Mazur-cl})
and the prime end boundary using the new definition
of prime ends introduced in \cite{ABBSprime}. 
\end{enumerate}

If $X$ is a quasiconvex proper metric space equipped with a doubling
measure $\mu$, then the following are true:
\begin{enumerate}
\setcounter{enumi}{2}
\item
If $\Om$ is a John domain, then 
$\Om$ is boundedly connected at the boundary. 
\item
If $\Om$ is an almost John domain, then 
$\Om$ is finitely connected at the boundary.
\end{enumerate}
\end{example}

For each $r>0$ let $\{\Gjrxo\}_{j=1}^{N(r,x_0)}$ be the components
of $B(x_0,r) \cap \Om$ that
have $x_0$ in their boundary,
i.e.\ $x_0 \in \itoverline{\Gjrxo}$. 
Here $N(r,x_0)$ is either a nonnegative integer or $\infty$, with $N(r,x_0)=0$ indicating that
there is no connected component of $B(x_0,r)\cap\Om$ that has $x_0$ in its boundary. Further, let
\[
    \Hrx{r}{x_0}=B(x_0,r) \cap \Om \setm
     \bigcup_{j=1}^{N(r,x_0)} \Gjrxo
\]
be the union of the remaining components (if any).
We will often drop $x_0$ from the notation when it is
clear from the context which point is under consideration.

\begin{lem} \label{lem-dec}
If $x_0 \notin \itoverline{\Hr{r}}$ for every $r>0$,
then $N(\,\cdot\,)$ is a nonincreasing function.
\end{lem}

Observe that if $0<r<s$ and $x_0 \notin \itoverline{\Hr{r}}$,
then $x_0 \notin \itoverline{\Hr{s}}$.

\begin{proof}
Let $0<r<s$. As $x_0 \notin \Hs$ there is at least one component $\Gones$.
Fix any component $\Gjs$. As $x_0 \in \itoverline{\Gjs}$, we have
that $\Gjs \cap B(x_0,r)$ is nonempty and thus consists of one or
more components. Not all of these components can lie in $\Hr{r}$ as
$x_0 \notin \itoverline{\Hr{r}}$. Hence there is $k(j)$ such that
$\Gj{k(j)}{r} \subset \Gjs \cap B(x_0,r)$.
As the $\Gjs$ are pairwise disjoint, $k(\,\cdot\,)$ must be injective,
and thus $N(s) \le N(r)$.
\end{proof}

\begin{prop} \label{prop1-fin}
The set\/ $\Om$ is finitely connected at $x_0$
if and only if for all $r>0$, $N(r)<\infty$ and $x_0 \notin \itoverline{\Hr{r}}$.
\end{prop}

As the following example illustrates,
having $0<N(r)<\infty$ for all $r>0$ by itself
does not guarantee finite connectivity.

\begin{example} \label{ex-comb-wide} (The topologist's comb I)
Let $\Om \subset \R^2$ be given by
\[
    \Om :=((-1,1) \times (0,2))
         \setm \bigl(\bigl\{\tfrac{1}{2},\tfrac{1}{3},\tfrac{1}{4},
         \ldots,0\bigr\}    \times (0,1]\bigr)
\]
and $x_0=(0,0)$. Then $N(r)=1$ for all $r>0$,
but $\Om$ is not finitely connected at $x_0$.
Note that $\Om$ is simply connected.
\end{example}

\begin{proof}[Proof of Proposition~\ref{prop1-fin}]
Assume first that $\Om$ is finitely connected at $x_0$,
and let $r>0$.
Then there is  a neighbourhood $G \subset B(x_0,r)$ of $x_0$
such that $G \cap \Om$ has only finitely many components.
It follows that $x_0  \notin \itoverline{\Hr{r}}$,
as otherwise $G \cap \Hr{r}$ would have infinitely many components.
Furthermore,
as $x_0 \in \itoverline{\Gjr}$ we must have that $G \cap \Gjr\ne \emptyset$
for every $j$. Moreover, two different components $\Gjr$ and $\Gkr$ cannot
intersect the same component of $G\cap\Om$, as $G \subset B(x_0,r)$.
Thus $N(r)$ is no larger than the number of components of $G$, which is
finite by assumption.

Conversely, let $r>0$ and assume that $N(r)<\infty$ and
$x_0 \notin \itoverline{\Hr{r}}$.
Then there is $0<s<r$ such that $B(x_0,s) \cap \Hr{r} = \emptyset$.
It follows that $G=B(x_0,s) \cup \bigcup_{j=1}^{N(r)} \Gjr \subset B(x_0,r)$
is a neighbourhood of $x_0$ such that
$G \cap \Om=\bigcup_{j=1}^{N(r)} \Gjr $ has
only $N(r)$ number of components.
\end{proof}

\begin{prop} \label{prop1-N}
The set\/ $\Om$ is $N$-connected at $x_0$
if and only if $x_0 \notin \itoverline{\Hr{r}}$ for all $r>0$
and\/ $\lim_{r \to 0} N(r) =N$. 
Moreover, in this case there is $r_0>0$ such that $N(r_0)=N$.
\end{prop}

\begin{proof}
The proof is similar to the proof of Proposition~\ref{prop1-fin}
with some extra details.
Assume first that $\Om$ is $N$-connected at $x_0$.
Just as in the proof of Proposition~\ref{prop1-fin}
we find that for all $r>0$ it is true that
$x_0 \notin \itoverline{\Hr{r}}$ and that $N(r) \le N$.
As $N(\,\cdot\,)$ is nonincreasing the latter condition is equivalent
to the condition $\lim_{r \to 0} N(r) \le N$. Since $N$ is the minimal such
integer, it follows that $\lim_{r\to 0}N(r)=N$.

Conversely, by Lemma~\ref{lem-dec}, $N(r)  \le \lim_{r \to 0} N(r)$
 for all $r>0$. Thus, the proof of Proposition~\ref{prop1-fin} shows that
in this case $\Om$ is at most $\lim_{r \to 0} N(r)$-connected
at $x_0$. It follows that we also must have $N=\lim_{r \to 0} N(r)$.
As $N(\,\cdot\,)$ is decreasing and integer-valued there must be
some $r_0>0$ such that $N(r_0)=N$.
\end{proof}

\begin{cor} \label{cor1-bdd}
The domain\/ $\Om$ is boundedly connected at $x_0$
if and only if $x_0 \notin \itoverline{\Hr{r}}$ for all $r>0$
and\/ $\lim_{r \to 0} N(r) < \infty$.
\end{cor}

This corollary follows directly from Proposition~\ref{prop1-N}.

\section{The Mazurkiewicz distance \texorpdfstring{$\dM$}{}}
\label{sect-dM}

\emph{From now on, we assume in this paper that $X$ is a \emph{proper} 
\textup{(}i.e.\
closed bounded sets are compact\/\textup{)} locally connected metric space.}

\medskip

Since $X$ is proper it is complete, and thus
Mazurkiewicz--Moore--Menger's theorem
shows that $X$ is locally pathconnected, see
Section~\ref{sect-lc}.

\begin{deff}
We define the \emph{Mazurkiewicz distance} $\dM$ on $\Om$ by
\[
     \dM(x,y) =\inf \diam E,
\]
where the infimum is taken over all connected sets $E \subset \Om$
containing $x,y \in \Om$.

For comparison, let us further define the \emph{inner metric}
$\din$ on $\Om$ by
\[
     \din(x,y)=\inf l_\ga,
\]
where the infimum is taken over all
rectifiable curves $\ga:[0,l_\ga] \to \Om$ parameterized by
arc length such that $\ga(0)=x$ and $\ga(l_\ga)=y$.
\end{deff}

The Mazurkiewicz distance $\dM$ is always a metric on $\Om$
and, because $X$ is locally connected, it generates the same topology as $d$ on $\Om$.
Furthermore, we have $d \le \dM \le \din$. The function $\din$ is a metric if and 
only if $\Om$ is rectifiably connected,
i.e.\ any two points are connected by a rectifiable curve
(a curve with finite length) in $\Om$. 

The Mazurkiewicz distance was introduced
by Mazurkiewicz~\cite{mazurkiewicz16},
 in relation to a classification of points on $n$-dimensional Euclidean continua.
The metric $\dM$ goes under different names in the literature,
it is e.g.\  denoted $\rho_A$ in \cite{mazurkiewicz16}, 
called \emph{relative distance}
and denoted $\varrho_r$ in Kuratowski~\cite{kuratowski2}, 
called \emph{Mazurkiewicz intrinsic metric} and denoted $\delta_D$ 
in Karmazin~\cite{karmazin2008},
denoted $\mu$ in Newman~\cite{newman}
(only in the first edition), and by $m$ in Ohtsuka~\cite{ohtsuka}.
In Aikawa--Hirata~\cite{AikawaHirata}, 
Freeman--Herron~\cite{FreemanHerron} 
and Herron--Sullivan~\cite{HerronSullivan} 
it is called \emph{inner diameter distance}.
Here and in
Adamowicz--Bj\"orn--Bj\"orn--Shan\-mu\-ga\-lin\-gam~\cite{ABBSprime}
and Bj\"orn--Bj\"orn--Shanmugalingam~\cite{BBSdir}
we call it the \emph{Mazurkiewicz distance}. 

\begin{deff}\label{Mazur-cl}
The completion of the metric space $(\Om,\dM)$ 
 is denoted $\overline{\Om}^M$.
Furthermore, the \emph{Mazurkiewicz boundary}
is $\partial_M\Om:=\overline{\Om}^M\setminus\Om$.
\end{deff}

Note that each point $y\in\partial_M\Om$ corresponds to 
(an equivalence class of) a sequence
$\{x_j\}_{j=1}^\infty$ which is a Cauchy sequence 
with respect to the Mazurkiewicz metric $d_M$ and 
which does not have
a limit point in $\Om$.

As advertised 
in Theorem~\ref{thm-clOmm-cpt}, 
the  closure $\clOmm$ is compact if and only if
$\Om$ is finitely connected at the boundary.
This result was proven by Karmazin~\cite[Theorem~1.3.8]{karmazin2008}, 
but the proof is
available only in Russian and not widely accessible. 
For the reader's convenience, we give a more elementary self-contained
proof which
appeals only to basic definitions.

\begin{proof}[Proof of Theorem~\ref{thm-clOmm-cpt}]
First assume that $\Om$ is finitely connected at the boundary.
As $\clOmm$ is a metric space, it is compact
if and only if it is sequentially compact,
see e.g.\ Munkres~\cite[Theorem~28.2]{munkres2}. 

Let $\{y_j\}_{j=1}^\infty$ be a sequence in $\clOmm$.
For each $j$ we can find $x_j \in \Om$ such that $\dM(x_j,y_j)<1/j$.
Then $\{x_j\}_{j=1}^\infty$ has a convergent subsequence if and only if
$\{y_j\}_{j=1}^\infty$ has, and it is thus enough to show that every sequence
$\{x_j\}_{j=1}^\infty$ from $\Om$ has a convergent subsequence in $\overline{\Om}^M$.
To this end, let $\{x_j\}_{j=1}^\infty$ be a sequence in $\Om$. Because $X$ is 
proper and $\Om$ is bounded, this sequence has a cluster point
$x_0 \in \overline{\Om}$ with respect to the original metric $d$.
By taking a subsequence if necessary we may assume that $x_j \doverto x_0$ as
$j \to \infty$. If $x_0 \in \Om$, then 
the local connectivity of $X$ implies that
$x_j \dMto x_0$ as $j \to\infty$,
and we are done. Assume therefore that $x_0 \in \bdy \Om$.

By assumption, $\Om$ is finitely connected at $x_0$, and in particular
$N(r)<\infty$ for all $r>0$.
There must therefore be some $k_0$ such that infinitely many of
the $x_j$ belong to $\Gj{k_0}{1}$.
Collect them in a subsequence
$\{x_j^{(0)}\}_{j=1}^\infty$.
Similarly, there is some $k_1$ such that infinitely many of the
$x_j^{(0)}$ belong to $\Gjbig{k_1}{\tfrac{1}{2}}$.
Again, collect them in a subsequence
$\{x_j^{(1)}\}_{j=1}^\infty$.
Proceeding inductively, we find $k_2$, $k_3, \ldots$, and
further and further refined subsequences $\{x_j^{(2)}\}_{j=1}^\infty$,
$\{x_j^{(3)}\}_{j=1}^\infty, \ldots$
such that $x_j^{(m)} \in \Gj{k_m}{2^{-m}}$.
A diagonalization argument lets us choose $z_j=x_j^{(j)}$, a subsequence
of the original sequence. Now, if $m>j$, then
\[
    \dM(z_j,z_m) \le \diam \Gj{k_j}{2^{-j}}
    \le \diam B(x_0,2^{-j}) \le 2^{1-j}.
\]
Hence $\{z_j\}_{j=1}^\infty$ is a Cauchy sequence in the metric
$\dM$ and thus has a limit in $\clOmm$.

Conversely, assume that $\Om$ is not finitely connected at the boundary.
Then there is some $x_0 \in \bdy \Om$ such that $\Om$ is not
finitely connected at $x_0$.
By Proposition~\ref{prop1-fin}, there is some $r>0$ such that either
$N(r)=\infty$ or $x_0 \in \itoverline{\Hr{r}}$.

If $N(r)=\infty$, then we can find $x_j \in \Gj{j}{r} \cap B\bigl(x_0,\frac{1}{2}r\bigr)$,
$j=1,2,\ldots\,$.
If $j \ne k$ and $x_j,x_k \in E \subset \Om$ for some connected set $E$,
then by construction $E$ must contain a point outside $B(x_0,r)$,
and hence $\diam E \ge r/2$ showing that
$\dM(x_j,x_k)\ge r/2$.
This implies that $\{x_j\}_{j=1}^\infty$ cannot have a convergent
subsequence with respect to $d_M$.

Finally, if $x_0 \in \itoverline{\Hr{r}}$, then
we can find points $x_j \in \Hr{r} \cap B(x_0,r/2)$
such that $x_j$ belongs to distinct components
of $\Hr{r}$ (and thus of $\Om \cap B(x_0,r)$).
The rest of the argument is identical to the reasoning of
the case $N(r)=\infty$ above.
\end{proof}

The following result and its proof give a good
understanding of how the boundary $\bdym \Om$ looks like.

\begin{prop}  \label{prop-N-conn}
Assume that\/ $\Om$ is $N$-connected at $x_0$.
Then there are exactly $N$ boundary points in $\bdym \Om$
corresponding to $x_0$.
\end{prop}

By saying that the point $y=\{x_j\}_{j=1}^\infty\in\bdym\Om$ 
corresponds to $x_0$, we mean that $x_j\doverto x_0\in\partial\Om$.

\begin{proof}
Without loss of generality we may assume that $N(1)=N$.
By Lemma~\ref{lem-dec}, $N(\,\cdot\,)$ is nonincreasing,
and so $N(r)=N$ for $0<r\le 1$.
For $0<r<1$ each $\Gjr$ belongs to one
of the components $\Gkone$, and we may assume that they
have been ordered so that $\Gjr\subset\Gjs$ if $0<r<s \le 1$.
For each positive integer $j \le N$, choose a
sequence $x_k \in \Gj{j}{1}$ such that $x_k \doverto x_0$.
As $x_0 \in \itoverline{\Gj{j}{1}}$ there are such sequences.
We shall show that $\{x_k\}_{k=1}^\infty$ has a limit in $\clOmm$.

Let $\eps>0$. Then by Proposition~\ref{prop1-N},
$x_0 \notin \itoverline{\Hr{\eps}}$. Therefore there is
$0 < \de<\eps$ such that
$B(x_0,\de) \cap \itoverline{\Hr{\eps}} = \emptyset$.
Moreover there is $K>1$ such that
$d(x_k,x_0) < \de$ for $k >K$.
Thus, for $k>K$, $x_k \in \Gj{j(k)}{\eps}$ for some
$j(k)$, and as $x_k \in \Gj{j}{1}$ we must have $j(k)=j$.
Hence for $k,l >K$,
\[
     \dM(x_k,x_l) \le \diam \Gj{j}{\eps} \le 2\eps.
\]
Thus $\{x_k\}_{k=1}^\infty$ is a Cauchy sequence
in the $\dM$-metric and has a limit in $\clOmm$.

As $\{x_k\}_{k=1}^\infty$ was an arbitrary sequence
in $\Gj{j}{1}$ with $x_k \doverto x_0$, all such sequences
must have the same limit in $\clOmm$, denoted $y_j$.

Now take
$\partial_M\Om\ni y=\{x_k\}_{k=1}^\infty$, where $\{x_k\}_{k=1}^\infty$ 
is a sequence in $\Om$ that is Cauchy with 
respect to the metric $\dM$ and in addition $x_k \doverto x_0$.
As $N$ is finite and $x_0 \notin \itoverline{\Hr{1}}$
there is an integer $j$ and a subsequence $\{x_{k_l}\}_{l=1}^\infty$
such that $x_{k_l} \in \Gj{j}{1}$ for all $l$.
It follows that $x_{k_l} \dMto y_j$.
Hence $y=y_j$ and $x_k \dMto y_j$.

Thus $y_1,\ldots,y_N$ are all the boundary points
in $\clOmm$ corresponding to $x_0$.
It remains to show that they are all distinct.
Let $1 \le j <k\le N$.
By the construction of $y_j$ we can find $x_j \in \Gj{j}{1}$
such that $\dM(x_j,y_j)< \tfrac{1}{4}$, and similarly
$x_k \in \Gj{k}{1}$ such that $\dM(x_k,y_k)< \tfrac{1}{4}$.
In particular $d(x_j,y_j)< \tfrac{1}{4}$
and $d(x_k,y_k)< \tfrac{1}{4}$.
Let $F \subset \Om$ be any connected open set containing
$x_j$ and $x_k$. As $x_j$ and $x_k$ belong to different components of
$B(x_0,1) \cap \Om$, the set
$F$ must contain a point $z \in \Om \setm B(x_0,1)$. Hence
\[
     \diam F  \ge d(z,x_j) \ge d(z,x_0) - d(x_0,x_j)
        > 1-\tfrac{1}{4} = \tfrac{3}{4},
\]
and thus $\dM(x_j,x_k) \ge \tfrac{3}{4}$. So
\[
       \dM(y_j,y_k) \ge \dM(x_j,x_k) - \dM(x_j,y_j) - \dM(x_k,y_k)
        > \tfrac{3}{4} - \tfrac{1}{4} - \tfrac{1}{4}
    = \tfrac{1}{4}.
    \qedhere
\]
\end{proof}

If $\Om$ is finitely but not boundedly connected
at $x_0$ one similarly gets infinitely many points
in $\bdym \Om$ corresponding to $x_0$.
The proof of this fact is similar to the proof above.

\medskip

We define $\Phi : \clOmm \to \overline{\Om}$ in the following way:
For $x \in \Om$ we let $\Phi(x)=x$.
Since $d \le \dM$, this is a $1$-Lipschitz
map on $\Om$. Because $\overline{\Om}$ is complete, 
this map has a unique continuous extension to $\clOmm$,
which we again denote by $\Phi$, which is
also $1$-Lipschitz. If $x_0 \in \bdy \Om$ and $\Om$ is $N$-connected at $x_0$, then
$\Phi^{-1}(x_0)$ consists of the $N$ points called $y_1,\ldots,y_N$ in the proof
above.

\section{Various types of local connectedness at the boundary}
\label{sect-loc-conn-bdry}

For a boundary point $x_0 \in \bdy \Om$ there are four types of
local connectedness one can discuss: that $\Om$ is locally connected
at $x_0$, that $\Om \cup \{x_0\}$ is locally connected at $x_0$,
that $\bdy \Om$ is locally connected at $x_0$ and that
$X \setm \Om$ is locally connected at $x_0$.
The last three conditions come in four variants according to
Definition~\ref{deff-lc}, while for the first we have just one variant
(given in Definition~\ref{deff-finite-conn}).
These properties are closely related and also closely related to
the compactness of $\clOmm$.

Let us start with the following characterization valid for
finitely connected planar domains.
We first recall the definition of local accessibility.

\begin{deff}
A boundary point $x_0 \in \bdy \Om$
is \emph{locally accessible} from $\Om$ if
for every $r>0$ there is $\de >0$ such that
for every $x \in B(x_0,\de)\cap \Om$
there is a (not necessarily rectifiable) curve 
$\ga : [0,1] \to X$ such that $\ga([0,1)) \subset B(x_0,r) \cap \Om$,
$\ga(0)=x$ and $\ga(1)=x_0$.
\end{deff}

See also the related definitions about connectivity given in
Definition~\ref{deff-lc}. 

The following theorem 
(except for \ref{it-compl})
was proven for simply connected planar domains by
Newman in the first edition of the book~\cite{newman}
(but was omitted from the
second edition of \cite{newman}).

\begin{thm} \label{thm-newman}
Assume that\/ $\Om \subset \R^2$ is a bounded and finitely  connected domain.
Then the following are equivalent\/\textup{:}
\begin{enumerate}
\item \label{it-locacc}
each  $x_0 \in \bdy \Om$
is locally accessible from\/ $\Om$\/\textup{;}
\item \label{it-locconn2}
for every $x_0 \in \bdy \Om$ we have that\/
$\Om \cup \{x_0\}$ is locally connected at $x_0$\/\textup{;}
\item \label{it-Hr}
for all $r>0$ and $x_0 \in \bdy\Om$ it is true that $x_0 \notin \itoverline{\Hrx{r}{x_0}}$\textup{;}
\item \label{it-finconn}
$\Om$ is finitely connected at the boundary\/\textup{;}
\item \label{it-cpt}
$\clOmm$ is compact\/\textup{;}
\item \label{it-locconn}
$\bdy \Om$ is locally connected\/\textup{;}
\item \label{it-compl}
$\R^2 \setm \Om$ is locally connected.
\end{enumerate}
\end{thm}

Since $\Om \subset \R^2$, it is 
\emph{finitely\/ \textup{(}simply\/\textup{)} connected} if and only if
$\Sphere^2\setm \Om$ has
only finitely many (one) components, where $\Sphere^2$ is the Riemann sphere.
Similar characterizations hold for the notions of countable, uncountable, and infinite connectedness
of plane domains.

Observe that we can equivalently rephrase  \ref{it-locconn2},
\ref{it-locconn} and \ref{it-compl}
using any of the other three types of local connectedness in
Definition~\ref{deff-lc}.
For \ref{it-locconn} and \ref{it-compl} the equivalence of the four
notions follows
from the Mazurkiewicz--Moore--Menger theorem, see Section~\ref{sect-lc}.
The Mazurkiewicz--Moore--Menger theorem is  however not 
available for \ref{it-locconn2}
since $\Om \cup \{x_0\}$ is not complete,
so in this case we instead have to appeal to
Theorem~\ref{thm-newman-pt} below.

In view of~\ref{it-finconn} and~\ref{it-locconn}
it is also natural to ask if the above statements could be
equivalent to the statement 
``$\overline{\Om}$ is locally connected'' as well.
That $\overline{\Om}$ is locally connected
 follows from the other statements, which
can be shown in the same way as the implication 
\ref{it-locconn-pt-lc} $\imp$ \ref{it-compl-pt-lc} in 
Theorem~\ref{thm-newman-pt}.
However, it does not imply the other statements, 
as shown by the topologist's comb in
Example~\ref{ex-comb-wide}. 
We therefore leave it out of the rest of the discussion.

Before proving Theorem~\ref{thm-newman}, let us consider
how much of this result is valid for more general domains $\Om$
(i.e.\ infinitely connected in $\R^2$ as well as general domains
in $\R^n$ and metric spaces).

\begin{thm} \label{thm-newman-gen}
Consider the same statements as
in Theorem~\ref{thm-newman} for a general bounded domain\/ $\Om$
in a proper locally connected metric space $X$
\textup{(}with\/ $\R^2$ replaced by $X$ in \ref{it-compl}\textup{)}.

Then the following implications are true\/\textup{:}
\[
    \ref{it-cpt}
    \ \Longleftrightarrow \  \ref{it-finconn}
    \ \Longrightarrow \  \ref{it-Hr}
    \ \Longleftrightarrow \  \ref{it-locconn2}
    \ \Longleftrightarrow \  \ref{it-locacc}
 \quad \text{and} \quad
       \ref{it-locconn} \ \Longrightarrow \  \ref{it-compl}.
\]

No other implication  is true even if we assume that\/ $\Om \subset \R^3$
is homeomorphic to a ball.
\end{thm}

In $\R^2$ a little more is true, even if $\Om$ is not simply connected.

\begin{thm} \label{thm-newman-R2}
Consider the same statements as
in Theorem~\ref{thm-newman} for an infinitely connected bounded domain\/ 
$\Om \subset \R^2$.
Then    \ref{it-locconn} and  \ref{it-compl} fail.

If\/ $\Om$ is moreover countably connected, i.e.\
$\R^2\setm\Om$ has only countably many components, then
\ref{it-locacc}--\ref{it-cpt}
are equivalent.
This holds also if\/ $\R^2\setm\Om$ consists of at most
countably many continua $K_j$,
$j=1,2,\ldots$, and\/ {\rm(}possibly uncountably many\/{\rm)} 
singleton components, and in addition  $\itoverline{E}\setm E$
is a union of a compact set and an at most countable set, 
where $E=\bigcup_{j=1}^\infty K_j$.
\end{thm}

Recall that a \emph{continuum} $K \subset \R^2$ 
is an infinite connected compact set.
We postpone the proof of Theorem~\ref{thm-newman-R2} until after
the proofs of Theorems~\ref{thm-newman} and~\ref{thm-newman-gen} later
in this section.

We do not know whether \ref{it-locacc} $\imp$ \ref{it-cpt}
always hold if $\Om$ is an uncountably connected planar domain.
The following example shows that it is possible for the statements 
\ref{it-locacc}--\ref{it-cpt} to be true
even if $\Om$ is infinitely connected.

\begin{example}
Let $C \subset [0,1]$ be the ternary Cantor set and $C_j$ be the $j$th generation
generating $C$ (i.e.\ $C=\bigcap_{j=0}^\infty C_j$ and $C_j$ consists
of $2^j$ intervals each of length $3^{-j}$).
Let further
\begin{align*}
    K& = (C \times \{0\}) \cup \bigcup_{j=0}^\infty (C_j \times \{2^{-j}\}), \\
    \Om& = (-2,2)^2 \setm K \quad \text{and} \quad 
    \Om'=\Om\setm([-1,0]\times \{0\}).
\end{align*}
Then $\Om$ and $\Om'$ are both uncountably connected domains
and locally connected or $2$-connected at all their boundary points.
Thus they both satisfy \ref{it-locacc}--\ref{it-cpt}.
Theorem~\ref{thm-newman-R2} is however only applicable
to $\Om$, but not to $\Om'$ since in that case
$\itoverline{E}\setm E = (C\setm\{0\})\times\{0\}$ 
is not a union of a compact set  and a countable set.

By replacing the line segments by thin rectangles it is easy
to  modify $\Om$ and $\Om'$ so that they become
locally connected at the boundary.
\end{example}

\begin{openprob}
Is Theorem~\ref{thm-newman-R2} also true
when $\itoverline{E}\setm E$ is not a union of a compact set
and an at most countable set?
\end{openprob}

If there are uncountably many continuum components of $\R^2 \setm \Om$,
then  uncountably many of them
have diameter larger than $\de$, for some $\de>0$, in which
case it seems difficult to fulfill \ref{it-locacc}.
This can be formulated as the following question.

\begin{openprob}
Does \ref{it-locacc} always fail if there
are uncountably many continuum components of $\R^2 \setm \Om$?
In view of Theorem~\ref{thm-newman-gen} it would follow
that \ref{it-locacc}--\ref{it-cpt} 
all fail, and are thus equivalent.
\end{openprob}

We now turn to the proof of Theorem~\ref{thm-newman-gen}.

\begin{proof}[Proof of Theorem~\ref{thm-newman-gen}]
The implications 
 \ref{it-finconn}  $\imp$ \ref{it-Hr}
    $\eqv$ \ref{it-locconn2}
    $\eqv$ \ref{it-locacc}
and
 \ref{it-locconn} $\imp$ \ref{it-compl}
follow directly from the corresponding
pointwise results in Theorem~\ref{thm-newman-pt} below.
The equivalence of \ref{it-finconn} and \ref{it-cpt} is the content of Theorem~\ref{thm-clOmm-cpt}.

Counterexamples for the failing implications are given
as examples below, but we list them here.
Note that they are all given for $\Om \subset \R^3$
which is homeomorphic to a ball.

\ref{it-finconn} $\not\imp$ \ref{it-compl}
This follows from Example~\ref{ex-R3-finconn-compl-new}.

\ref{it-compl} $\not\imp$ \ref{it-locconn}
This follows from Example~\ref{ex-R3-compl-locconn}.

\ref{it-locconn} $\not\imp$ \ref{it-locacc}
This follows from Example~\ref{ex-R3-locconn-locacc}.

\ref{it-locacc} $\not\imp$ \ref{it-finconn}
This follows from Example~\ref{ex-Jana}.

The failure of the remaining implications follows 
from these counterexamples 
in combination with the implications obtained in the beginning of the proof.
\end{proof}

The implication \ref{it-finconn}
$\imp$ \ref{it-compl} is in general false.
The following example shows that
if $\Om \subset \R^3$,
then it is not enough to assume that
$\Om$ is homeomorphic to a ball.
Note that $\Om$ below is locally
connected at the boundary.

\begin{example} \label{ex-R3-finconn-compl-new}
Let $G \subset \R^2$ be given by
\[
    G := ((0,1) \times (-1,1))
         \setm \bigcup_{j=1}^\infty  \itoverline{B((2^{-j},0),2^{-j}/10)}.
\]
Further, let
\[
   \Om=((0,1) \times (-1,1)^2) \setm (G \times [0,1)).
\]
Then $\Om$ is locally connected at the boundary,
while $X \setm  \Om$ is not locally connected at $(0,0,0)$.
Note that $\Om$ is homeomorphic to a ball.
\end{example}

In $\R^3$ we have the following example
showing that \ref{it-compl} $\not\imp$ \ref{it-locconn}.
(It also shows that \ref{it-compl} $\not\imp$ \ref{it-locacc}.)

\begin{example} \label{ex-R3-compl-locconn}
Let
\[
    \Om=(0,1)^3
    \cup \bigcup_{j=1}^\infty ((2^{-2j-1},2^{-2j})^2 \times (0,2)).
\]
In this case $\R^3 \setm \Om$ is locally connected at every point, while
$\bdy \Om$ is not locally connected at the points in
$A=\{(0,0,z): 1 < z \le 2\}$.
Moreover the points in $A$ are not accessible from $\Om$.
Note that $\Om$ is homeomorphic to a ball.
\end{example}

The following example shows that
\ref{it-locconn} $\not\imp$ \ref{it-locacc} in $\R^3$.

\begin{example} \label{ex-R3-locconn-locacc}
Let $K \subset \R^2$ be given by 
\[
K=\bigcup_{j=1}^\infty(([-2^{-j},2^{-j}]\times\{0,2^{-j}\})\cup(\{-2^{-j},2^{-j}\}\times[0,2^{-j}])),
\]
and let $\Om \subset \R^3$ be given by
\[
    \Om=(-1,1)^3\setm (K\times[0,1)).
\]
Then $\bdy\Om$ is locally connected, 
but no point $(0,0,z)\in\bdy\Om$ is locally accessible from
$\Om$ when $0<z \le 1$. Moreover, $\Om$ is not finitely connected at these points.
Note that $\Om$ is homeomorphic to a ball.
\end{example}

By changing the shape of the removed set, we obtain the
following example which shows that
\ref{it-locacc} $\not\imp$ \ref{it-finconn} in $\R^3$.

\begin{example} \label{ex-Jana}
Let $K \subset \R^2$ be given by
\[
     K=\bigcup_{j=2}^\infty \bdy B((0,2^{-j}),2^{-j}),
\]
and let $\Om \subset \R^3$ be given by
\[
    \Om=(-1,1)^3 \setm (K \times [0,1)).
\]
Then any boundary point $x_0 \in \bdy \Om$ is
locally accessible from $\Om$,
but $\Om$ is not finitely connected at 
any boundary point $(0,0,z)$ with $0<z\le1$. 
Note that $\Om$ is homeomorphic to a ball.
\end{example}

\begin{proof}[Proof of Theorem~\ref{thm-newman}.]
    \ref{it-cpt} $\eqv$ \ref{it-finconn} $\imp$ \ref{it-Hr}
     $\eqv$ \ref{it-locconn2} $\eqv$ \ref{it-locacc}
and
   \ref{it-locconn} $\imp$ \ref{it-compl}
These implications follow from
Theorem~\ref{thm-newman-gen} 
(or more directly from Theorems~\ref{thm-clOmm-cpt} and~\ref{thm-newman-pt}).

\medskip

\emph{Case} 1. The special case when $\Om$ is simply connected.

\ref{it-locacc} $\imp$ \ref{it-cpt} This is proved in Newman~\cite[p.~183]{newman}
(but does not appear in the second edition of \cite{newman}).

\ref{it-locacc} $\eqv$ \ref{it-locconn} 
This is in Newman~\cite[p.~187, Theorem~12.2]{newman}
(but not in the second edition of \cite{newman}).

\ref{it-compl} $\imp$   \ref{it-locconn}
This is part of Pommerenke~\cite[p.~20, Theorem~2.1]{pommerenke}.

\medskip

\emph{Case} 2. The general case when $\Om$ is finitely connected,
i.e.~$\Sphere^2 \setm \Om$ has finitely many components.

Let $K$ be one of the components, which must be
at a positive
distance from the other components. 
Since the conditions \ref{it-locacc}--\ref{it-finconn},
\ref{it-locconn} and \ref{it-compl} are all local,
any of these conditions holds for $x \in K$ with respect
to $\Om$ if and only if it holds with respect to 
$\Om':=\Sphere^2 \setm K$.
As $\Om'$ is simply connected, and we have already
obtained the equivalence for the simply connected case,
we see that 
\ref{it-locacc}--\ref{it-finconn},
\ref{it-locconn} and \ref{it-compl}
all are equivalent
also
in the finitely connected case.
That \ref{it-cpt} is equivalent to the other statements
follows from Theorem~\ref{thm-clOmm-cpt}.
\end{proof}

\begin{proof}[Proof of Theorem~\ref{thm-newman-R2}]
Let us first show that
   \ref{it-locconn} and  \ref{it-compl} are false.
Since $\Om$ is bounded and infinitely connected,
$\R^2 \setm \Om$ consists of infinitely many bounded components (and one
unbounded component). Furthermore, we can choose one point from each of these
components to obtain a bounded infinite collection of points. This collection must
have an accumulation point $x \in \R^2$.
But then $\R^2 \setm \Om$ is not locally connected at $x$.
It also follows that  $\bdy \Om$ is not locally connected at $x$.

\ref{it-locacc} $\imp$ \ref{it-finconn} 
This implication is quite complicated to prove and we
 postpone it to the next section, where it is obtained
as Theorem~\ref{thm-newman-count-singleton}.

The remaining implications follow from Theorem~\ref{thm-newman-gen}.
\end{proof}

It is also of interest to consider the corresponding
pointwise statements.
Contrary to the global statements we need to consider the different
types of local connectedness.
The statement \ref{it-cpt} above does not have a local version,
instead we consider the statement \ref{it-locconn3-pt} below.

\begin{thm} \label{thm-newman-pt}
Assume that $X$ is a proper locally connected metric space
and that\/ $\Om \subset X$ is open.
Let $x_0 \in \bdy \Om$.
Consider the following statements\/\textup{:}
\begin{enumerate}
\renewcommand{\theenumi}{\textup{(\alph{enumi}$'$)}}%
\item \label{it-locacc-pt}
$x_0$
is locally accessible from\/ $\Om$\/\textup{;}
\setcounter{saveenumi}{\value{enumi}}
\stepcounter{saveenumi}
\setcounter{enumi}{0}
\renewcommand{\theenumi}{\textup{(\alph{saveenumi}\arabic{enumi}$'$)}}%
\item \label{it-locconn2-pt-lpc}
$\Om \cup \{x_0\}$ is locally pathconnected at $x_0$\/\textup{;}
\item \label{it-locconn2-pt-lc}
$\Om \cup \{x_0\}$ is locally connected at $x_0$\/\textup{;}
\item \label{it-locconn2-pt-pk}
$\Om \cup \{x_0\}$ is pathconnected im kleinen at $x_0$\/\textup{;}
\item \label{it-locconn2-pt-k}
$\Om \cup \{x_0\}$ is connected im kleinen at $x_0$\/\textup{;}
\renewcommand{\theenumi}{\textup{(\alph{enumi}$'$)}}%
\setcounter{enumi}{\value{saveenumi}}
\item \label{it-Hr-pt}
for all $r>0$ it is true that $x_0 \notin \itoverline{\Hr{r}}$\/\textup{;}
\item \label{it-finconn-pt}
$\Om$ is finitely connected at $x_0$\/\textup{;}
\item \label{it-locconn3-pt}
$\Om$ is locally connected at $x_0$\/\textup{;}
\setcounter{saveenumi}{\value{enumi}}
\stepcounter{saveenumi}
\setcounter{enumi}{0}
\renewcommand{\theenumi}{\textup{(\alph{saveenumi}\arabic{enumi}$'$)}}%
\item \label{it-locconn-pt-lpc}
$\bdy \Om$ is locally pathconnected at $x_0$\/\textup{;}
\item \label{it-locconn-pt-lc}
$\bdy \Om$ is locally connected at $x_0$\/\textup{;}
\item \label{it-locconn-pt-pk}
$\bdy \Om$ is pathconnected im kleinen at $x_0$\/\textup{;}
\item \label{it-locconn-pt-k}
$\bdy \Om$ is connected im kleinen at $x_0$\/\textup{;}
\stepcounter{saveenumi}
\setcounter{enumi}{0}
\item \label{it-compl-pt-lpc}
$X \setm \Om$ is locally pathconnected at $x_0$\/\textup{;}
\item \label{it-compl-pt-lc}
$X \setm \Om$ is locally connected at $x_0$\/\textup{;}
\item \label{it-compl-pt-pk}
$X \setm \Om$ is pathconnected im kleinen at $x_0$\/\textup{;}
\item \label{it-compl-pt-k}
$X \setm \Om$ is connected im kleinen at $x_0$.
\end{enumerate}
Then the following implications are true\/\textup{:}
\[
    \ref{it-locconn3-pt}
    \smallimp 
    \ref{it-finconn-pt}
    \smallimp \ref{it-Hr-pt}
    \smalleqv \ref{it-locconn2-pt-lpc}
    \smalleqv \ref{it-locconn2-pt-lc}
    \smalleqv \ref{it-locconn2-pt-pk}
    \smalleqv \ref{it-locconn2-pt-k}
    \smalleqv \ref{it-locacc-pt}
\]
and
\[
\xymatrix{
     \text{\ref{it-locconn-pt-lpc}} \ar@{=>}[rrr] \ar@{=>}[ddd] \ar@{=>}[rd]
          &&&  \text{\ref{it-locconn-pt-lc}}\ar@{=>}[ddd] \ar@{=>}[ld] \\
     & \text{\ref{it-compl-pt-lpc}} \ar@{=>}[d] \ar@{=>}[r]
             & \text{\ref{it-compl-pt-lc}} \ar@{=>}[d] \\
     & \text{\ref{it-compl-pt-pk}} \ar@{=>}[r] & \text{\ref{it-compl-pt-k}} \\
     \text{\ref{it-locconn-pt-pk}} \ar@{=>}[rrr] \ar@{=>}[ru]
       &&& \text{\ref{it-locconn-pt-k}.} \ar@{=>}[lu]
}
\]
No other implication is true even if we assume
that\/ $\Om$ is a simply connected subset of\/ $\R^2$.
\end{thm}

\begin{proof}[Proof of Theorem~\ref{thm-newman-pt}]

\ref{it-locacc-pt}  $\imp$  \ref{it-locconn2-pt-lpc}
 If $x_0$ is locally accessible from $\Om$, then 
for every $r>0$ there exists a positive number $\delta$ such that each 
$y\in B(x_0,\delta)\cap\Om$ can
be connected by a curve $\gamma_y$ in $B(x_0,r)\cap\Om$ to $x_0$. 
If $y_1,y_2\in B(x_0,\delta)\cap\Om$,
then the concatenation (at $x_0$) of the two curves $\gamma_{y_1}$ and $\gamma_{y_2}$
gives a curve in $B(x_0,r)\cap(\Om\cup\{x_0\})$ connecting $y_1$ with $y_2$. 
It follows that 
the pathcomponent $x_0$ belongs to in $B(x_0,r)\cap(\Om\cup\{x_0\})$ contains 
$B(x_0,\delta)\cap(\Om\cup\{x_0\})$.
By the local pathconnectedness of $X$, it now follows that this pathcomponent is a
relatively open subset of $\Om\cup\{x_0\}$, whence we know that 
$\Om \cup \{x_0\}$ is
locally pathconnected at $x_0$. 

\ref{it-locconn2-pt-lpc}
$\imp$ \ref{it-locconn2-pt-lc}
$\imp$ \ref{it-locconn2-pt-k} and
\ref{it-locconn2-pt-lpc}
$\imp$ \ref{it-locconn2-pt-pk}
$\imp$ \ref{it-locconn2-pt-k}
These implications follow from \eqref{eq-lc-trivial}.

\ref{it-locconn2-pt-k} $\imp$ \ref{it-locacc-pt}
This follows from Whyburn~\cite[p.~111, Theorem~4.1]{whyburn42}
(which should be applied to the component of $X$ containing $x_0$).
Observe that Whyburn's definition of local connectedness is the same
as our definition of connectedness im kleinen. 

$\neg$ \ref{it-Hr-pt} $\imp$ $\neg$ \ref{it-locconn2-pt-lc}
Let $r>0$ be such that  $x_0 \in \itoverline{\Hr{r}}$,
and let $G \subset B(x_0,r)$
be a neighbourhood of $x_0$ in the relative topology of $\Om \cup \{x_0\}$.
Then $G \cap \Hr{r}$ must have infinitely many components,
none of which has $x_0$ in its boundary.
Each of these components is also a component in $G$, which
is thus not connected. 
Hence
we have shown that $B(x_0,r)\cap(\Om\cup\{x_0\})$
does not have a (relatively) open connected subset containing $x_0$.

\ref{it-Hr-pt} $\imp$ \ref{it-locconn2-pt-lc}
Fix $s>0$. Then by hypothesis $x_0\notin \itoverline{\Hs}$ and hence there
is $0<t<s$ such that $B(x_0,t) \cap \Hs = \emptyset$.
Let $G=\bigcup_{j=1}^{N(s)} \Gjs$
(where $N(s)$ is either a positive integer or $\infty$).
Then
\[
    B(x_0,t) \cap \Om \subset G \subset
    B(x_0,s) \cap \Om.
\]
So $G'=G\cup\{x_0\}$ is a neighbourhood of $x_0$ in
$\Om \cup \{x_0\}$. We now show that $G'$ is  connected.
Since $x_0$ is in the boundary of each $\Gjs$ and $\Gjs$ is 
connected, it follows that $\Gjs\cup\{x_0\}$ is connected. 
As $G'$ is the
union of a collection of connected sets, whose intersection contains $x_0$,
we see that $G'$ is connected as well,
see e.g.\ Lemma~\ref{lem-connected}.

 \ref{it-finconn-pt} $\imp$ \ref{it-Hr-pt} 
This follows from Proposition~\ref{prop1-fin}.

\ref{it-locconn3-pt} $\imp$ \ref{it-finconn-pt}
This follows directly from the definition,
and we have thus completed the proof of the first set of 
relations claimed in this theorem.

The implications
\ref{it-locconn-pt-lpc}
$\imp$ \ref{it-locconn-pt-lc}
$\imp$ \ref{it-locconn-pt-k},
\ref{it-locconn-pt-lpc}
$\imp$ \ref{it-locconn-pt-pk}
$\imp$ \ref{it-locconn-pt-k},
\ref{it-compl-pt-lpc}
$\imp$ \ref{it-compl-pt-lc}
$\imp$ \ref{it-compl-pt-k}
and
\ref{it-compl-pt-lpc}
$\imp$ \ref{it-compl-pt-pk}
$\imp$ \ref{it-compl-pt-k}
follow from \eqref{eq-lc-trivial}.

\ref{it-locconn-pt-lpc} $\imp$ \ref{it-compl-pt-lpc}
Let $r>0$. Then there is an open  set $G$, with $x_0\in G \subset B(x_0,r)$, 
such that $G \cap \bdy \Om$ is pathconnected.
Let $\Gt$ be the pathcomponent of $G$ containing $G \cap \bdy \Om$.
As $X$ is locally connected, $\Gt$ is open and moreover contains $x_0$.
Let $F= \Gt \setm \Om$.
We now show that $F$ is pathconnected.
Let $a,b\in F$.
Then there is a path $\gamma$ in $\Gt$ connecting $a$ to $b$. If
$\gamma$ does not intersect $\Om$, then $\gamma$ lies in $F$, and so there is a 
path in $F$ connecting $a$ to $b$. If $\gamma$ intersects $\Om$, then it intersects
$\bdy\Om$. 
Let $z$ and $w$ be the first and last time $\gamma$ 
intersects the closed set $\bdy\Om$. 
Then $z,w\in\Gt\cap\bdy\Om=F\cap\bdy\Om$, and
so by the pathconnectedness of $G\cap\bdy\Om=\Gt\cap\bdy\Om$ 
there is a path $\beta$ connecting
$z$ to $w$ in $\Gt\cap\bdy\Om\subset F$. 
Let $\gamma_z$ and $\gamma_w$ be the 
subcurves of $\gamma$ that lie in $F$ and connect $a$ to $z$ and  $b$ to $w$ 
respectively.
The concatenation of $\gamma_z$, $\beta$ and $\gamma_w$ is a path in $F$ connecting
$a$ to $b$. Hence $F$ is pathconnected and relatively open in $X\setm\Om$,
and so $X \setm \Om$ is
locally pathconnected at $x_0 \in \bdy \Om$.

\ref{it-locconn-pt-lc} $\imp$ \ref{it-compl-pt-lc}
Let $r>0$. Then there is an open  set $G$, with $x_0\in G \subset B(x_0,r)$, 
such that $G \cap \bdy \Om$ is connected.
Let $\Gt$ be the component of $G$ containing $G \cap \bdy \Om$.
As $X$ is locally connected, $\Gt$ is open and moreover contains $x_0$.
Let $F= \Gt \setm \Om$. We now show that $F$ is connected. Suppose that
$F=A_1 \cup A_2$, where $A_1$ and $A_2$
are disjoint open sets (in the relative topology of $F$).
Let $F_1=A_1 \cap \bdy \Om$ and  $F_2=A_2 \cap \bdy \Om$.
Then 
\[
     F_1 \cup F_2 = \Gt \cap \bdy \Om = G \cap \bdy \Om.
\]
Moreover $F_1$ and $F_2$ are disjoint and open in 
the relative topology of 
$F \cap \bdy \Om = G \cap \bdy \Om$
which is connected.
Hence one of them must be empty, say $F_2 = \emptyset$.
It follows that $A_2$ and
$A_1':=A_1 \cup (\Om \cap \Gt)$ must be open in the original topology of $X$.
As $A_1'$ and $A_2$ are open and disjoint while
$\Gt=A_1'\cup A_2$ is connected, one of them
must be empty. Since $x_0\in F_1\subset A_1'$, we have $A_2$  empty. 
Thus $F$ is connected and $X \setm \Om$ is locally
connected at $x_0 \in \bdy \Om$.

\ref{it-locconn-pt-pk} $\imp$ \ref{it-compl-pt-pk}
In this case, for each $r>0$ there is a
set $E$ and $\delta>0$ such that $B(x_0,\delta)\subset E\subset B(x_0,r)$ and 
$E\cap\bdy\Om$ is pathconnected. As before, let $\Et\subset X$ be the pathcomponent
of $E$ containing $x_0$ (and hence $E\cap\bdy\Om$), and let $F=\Et\setm\Om$.
As in the proof of \ref{it-locconn-pt-lpc} $\imp$ \ref{it-compl-pt-lpc}, 
we see that $F$ is
pathconnected. 
Since $X$ is locally pathconnected, we see that $\Et$ contains a neighbourhood
of $x_0$, and 
thus $F$ contains a 
neighbourhood of $x_0$ (in the relative topology of $X\setm\Om$). 
Hence $X \setm \Om$ 
is pathconnected im kleinen at $x_0$.

\ref{it-locconn-pt-k} $\imp$ \ref{it-compl-pt-k} 
The proof of this implication is very similar
to the proof of \ref{it-locconn-pt-pk} $\imp$ \ref{it-compl-pt-pk};
just drop the prefix \emph{path} (from all occurrences) and appeal to 
\ref{it-locconn-pt-lc} $\imp$ \ref{it-compl-pt-lc}
instead of \ref{it-locconn-pt-lpc} $\imp$ \ref{it-compl-pt-lpc}.

Let us now turn to the counterexamples for the other implications.
We present most of them as examples below, but list them here.
Note that all counterexamples are with simply connected
$\Om \subset \R^2$.

\ref{it-locacc-pt} $\not\imp$ \ref{it-finconn-pt}
This follows from Example~\ref{ex1}.

\ref{it-finconn-pt} $\not\imp$ \ref{it-locconn3-pt}
This follows from Example~\ref{ex-slit-disc}.

\ref{it-locconn3-pt} $\not\imp$ \ref{it-compl-pt-k}
This follows from Example~\ref{ex-comb} with $x_0=x_1$.

\ref{it-locconn-pt-lpc} $\not\imp$ \ref{it-locconn2-pt-lc}
This follows from Example~\ref{ex-comb} with $x_0=x_2$.

\ref{it-compl-pt-lpc} $\not\imp$ \ref{it-locconn-pt-k}
This follows from Example~\ref{ex-thick-comb}.

\ref{it-locconn-pt-pk} $\not\imp$ \ref{it-locconn-pt-lc}
and
\ref{it-compl-pt-pk} $\not\imp$ \ref{it-compl-pt-lc}
This follows from Example~\ref{ex-Rempe}.

\ref{it-locconn-pt-lc} $\not\imp$ \ref{it-locconn-pt-pk}
and
\ref{it-compl-pt-lc} $\not\imp$ \ref{it-compl-pt-pk}
This follows from
Example~\ref{ex-iterated-top-sine}.

The failure of the remaining implications follows form the above examples 
in combination with the implications obtained in the beginning of the proof.
\end{proof}

\begin{example} \label{ex1}
Here we use complex notation.
Let $\Om \subset \R^2$ be given by
\[
    \Om :=(0,2)^2
         \setm \bigcup_{j=1}^\infty
    \{re^{i/j}: 0 \le r\le 1\}.
\]
Then $(0,0)$ is locally accessible from $\Om$,
while $\Om$ is not finitely connected at $0$.
Note that $\Om$ is simply connected.
\end{example}

\begin{example} \label{ex-slit-disc}
Let $\Om$ be the \emph{slit disc} $B((0,0),1) \setm ((-1,0]\times \{0\})
\subset  \R^2$. 
Then $\Om$ is finitely connected at the boundary but not locally connected
at the boundary.
Note that $\Om$ is simply connected.
\end{example}

\begin{example} \label{ex-comb} (The topologist's comb II)
Let $\Om \subset \R^2$ be given by
\[
    \Om :=(0,2)^2
         \setm \bigl(\bigl\{1,\tfrac{1}{2},\tfrac{1}{3},
         \ldots,\bigr\}
    \times (0,1]\bigr),
\]
$x_1=(0,1)$ and $x_2=(0,0)$.
Then $\Om$ is locally connected at $x_1$,
while $X \setm \Om$ is not connected im kleinen at $x_1$.
Moreover, $\Om\cup\{x_2\}$ is not locally connected at $x_2$,
while $\bdy \Om$ is locally pathconnected at $x_2$.
Note that $\Om$ is simply connected.
\end{example}

\begin{example} \label{ex-thick-comb}
Let $\Om \subset \R^2$ be given by
\[
    \Om :=(0,2)^2
         \setm \biggl( (0,1] \times \bigcup_{j=1}^\infty [2^{-2j-1},2^{-2j}]
    \biggr).
\]
Then $X \setm \Om$ is locally pathconnected at $(0,0)$,
while $\bdy \Om$ is not connected im kleinen at $(0,0)$.
Note that $\Om$ is simply connected.
\end{example}

\begin{example} \label{ex-Rempe}
Let $L(a,b)$ be the closed line segment between $a$ and $b$ in $\R^2$.
Let
\[
    E=L((0,0),(1,0)) \cup\bigcup_{j,k=0}^\infty L((2^{-j-1},2^{-j-k}),(2^{-j},0))
    \quad \text{and} \quad \Om=[-1,1]^2  \setm E.
\]
Then $\bdy\Om$ and $X \setm \Om$ are pathconnected im kleinen at $(0,0)$,
while neither of them is locally connected at $(0,0)$.
Note that $\Om$ is simply connected and its boundary
consists of a chain of "brooms" near the origin. 
This example can, e.g., be found in  
Munkres~\cite[Figure~25.1]{munkres2}.
\end{example}

\begin{example} \label{ex-iterated-top-sine}
For $k=1,2,\ldots,$ let
\begin{align*}
     E_k & =     \{(x,4^{-k}
       \sin (\pi/(4^kx-1))) : 4^{-k}<x \le 2 \cdot 4^{-k}\},
     \\
    F & =\{(0,0)\} \cup \bigcup_{j=1}^\infty \bigl(E_j
    \cup \bigl(\bigl[\tfrac{1}{2}\cdot 4^{-j},4^{-j}] \times \{0\}\bigr)
    \cup (\{4^{-j}\} \times [-4^{-j},4^{-j}])\bigr), \\
    \Om & = [-1,1]^2 \setm F.
\end{align*}
Then $\bdy\Om$ and $X \setm \Om$ are locally connected  at $(0,0)$,
while neither of them is pathconnected im kleinen at $(0,0)$.
Note that $\Om$ is simply connected.
\end{example}

\section{Countably connected planar domains}
\label{sect-count-connect}

In this section, we shall prove the implication 
\ref{it-locacc} $\imp$ \ref{it-finconn} 
(local accessibility of boundary points from $\Om$
implies finite connectedness of $\Om$ at the boundary)
of Theorem~\ref{thm-newman-R2},
viz.\ Theorem~\ref{thm-newman-count-singleton}.

Note that Example~\ref{ex-Jana} shows that the implication
\ref{it-locacc} $\imp$ \ref{it-finconn} can fail even if $\Om\subset\R^3$
is homeomorphic to a ball.
Thus, the arguments in this section cannot be 
 applied to the  higher-dimensional case.

The following general topological lemma provides us with a sufficient condition for 
connectedness of unions of connected sets,  and in fact we have already used a 
special case  of
it in the proof of Theorem~\ref{thm-newman-pt}.

\begin{lem} \label{lem-connected}
Let $E_0$ and $E_\al$, $\al\in A$, be connected sets.
Assume that  $E_0\cap\itoverline{E}_\al$
is nonempty for every $\al\in A$.
Then $E=E_0\cup\bigcup_{\al\in A}E_\al$ is connected.
\end{lem}

This lemma can be found in Moore~\cite[Theorem~I.30, p.\ 11]{moore62}, 
but in the spirit of 
a discussion using modern definitions, we provide a 
proof of the lemma below.

We alert the reader that the terminology in Moore~\cite{moore62} is
somewhat nonstandard. 
In particular, by compact sets \cite{moore62} means sequentially precompact sets 
(which for $\R^2$ reduces to bounded sets),
and a continuum is a closed connected set 
(not necessarily compact nor nondegenerate). 
We will not use such archaic terminology.

\begin{proof} 
Let $x \in E_0$ and let $F$ be the component of $E$ containing $x$.
As $E_0$ is connected we must have $E_0 \subset F$.
Next, pick an arbitrary $\alp \in A$.
Since $E_\alp$ is connected, its closure relative to $E$
is also connected. 
This closure contains a point $z \in E_0$ and hence must be contained
in  the component of $z$, which is $F$.
In particular $E_\alp \subset F$.
Since $\alp$ was arbitrary we see that $E=F$ and hence $E$ is connected.
\end{proof}

In order to obtain our results in this section we
deal with upper semicontinuous
collections, which we now define.

\begin{deff}
Let $\G$ be a collection of pairwise disjoint compact subsets of $\Sp^2$.
Then $\G$ is \emph{upper semicontinuous} if for every sequence 
$\{K_j\}_{j=1}^\infty \subset\G$ and every set $K \in \G$ it is true
that if  $x_j,y_j\in K_j$,
$j=1,2,\ldots$, are two sequences such that $x_j\to x$ as $j\to\infty$
for some $x \in K$, 
then there exists a subsequence $\{y_{j_k}\}_{k=1}^\infty$ of $\{y_j\}_{j=1}^\infty$
and some $y \in K$ such that $y_{j_k}\to y$ as $k \to \infty$.
\end{deff}

We will need the following sufficient conditions for upper semicontinuous
collections.

\begin{thm}   \label{thm-comp-usc}
The following are true\/\textup{:}
\begin{enumerate}
\item  \label{it-comp-usc-a}
The components of a compact subset $F$ of\/ $\Sp^2$
 form an upper semicontinuous collection.
\item \label{it-comp-usc-b}
A subcollection of an upper semicontinuous collection is 
upper semicontinuous.
\item \label{it-comp-usc-c}
 If $\mathcal{G}$ is an upper semicontinuous collection 
and $K'$ is a compact set 
such that  
\[
K'\cap\bigcup_{K\in\mathcal{G}} K=\emptyset 
\quad \text{and} \quad  
K'\cup\bigcup_{K\in\mathcal{G}} K \quad \text{is compact,}
\]
then $\mathcal{G}\cup\{K'\}$
is upper semicontinuous.
\end{enumerate}
\end{thm}

Part \ref{it-comp-usc-a} 
can be found in Moore~\cite[Theorem~V.1.20, p.\ 284]{moore62},
but we include a proof here using modern terminology.

\begin{proof}
We will first prove \ref{it-comp-usc-a}. To do so
we employ the fact that an infinite sequence of compact 
subsets of $\Sp^2$
has a subsequence that converges in the Hausdorff topology to a 
compact set, see e.g.\ Lemma 5.31 in Bridson--Haefliger~\cite{BH}.

Let $\{K_j\}_{j=1}^\infty \subset\G$ and $K \in \G$,
where $\G$ is the collection of the components of $F$.
Next, let $x_j,y_j\in K_j$,
$j=1,2,\ldots$, be two sequences such that $x_j\to x$ as $j\to\infty$, 
for some $x \in K$.

By passing to a subsequence if necessary, we may assume that $K_j\to L$ in the 
Hausdorff topology for some compact set $L$. We now show that $L$ has to be connected. 
If $L$ is not connected, then there are two open sets $U$ and $V$ such that $L\subset U\cup V$, $L\cap U$ and
$L\cap V$ are nonempty but $L\cap U\cap V$ is empty. Since $L$ is compact and $U$ 
and $V$
are open, it follows that $L\setm V=L\cap U=:A$ and $L\setm U=L\cap V=:B$ are compact and
disjoint. Hence there is some $\eps>0$ such that the $\eps$-neighbourhoods
of $A$ and 
$B$ are disjoint. On the other hand, since $K_j\to L$ in the Hausdorff topology, 
for sufficiently large $j$ we have that $K_j$ is contained in the $\eps$-neighbourhood of 
$L=A\cup B$, which violates the fact that $K_j$ is connected. Thus we conclude that $L$ is also connected.

Now we have $x\in L\cap K$, which as a consequence means that $L\cup K$ is also connected, by Lemma~\ref{lem-connected}.
Since $K$ is a component of $F$ and $L\subset F$ (because $F$ is closed),
it follows that $L\subset K$.
On the other hand, by passing to a further subsequence if necessary and 
using that $F$ is compact
we may assume that $y_j\to y$ for some $y\in F$. 
Thus $y\in L \subset K$
showing that the collection is upper semicontinous.

The claim \ref{it-comp-usc-b} is easily seen to be true.

To prove \ref{it-comp-usc-c}, note that the upper semicontinuity condition is satisfied if the "limiting set" 
$K\in\mathcal{G}$. 
Now suppose that $K_j$, $j=1,2,\cdots,$ 
are sets from the collection $\mathcal{G}$
and that for each $j$ there is a point $x_j\in K_j$ such that 
$x_j\to x_\infty\in K'$. If $y_j\in K_j$ as well, then
because $K'\cup\bigcup_{K\in\mathcal{G}}K$ is compact, 
it follows that there is a subsequence 
$\{y_{j_k}\}_{k=1}^\infty$ and a point $y_\infty\in K'\cup\bigcup_{K\in\mathcal{G}}K$ with
$y_{j_k}\to y_\infty$ as $k\to\infty$. 
If $y_\infty\in K$ for some $K\in\mathcal{G}$, then by the
upper semicontinuity of the collection $\mathcal{G}$ we have that $x_\infty\in K$, which violates the assumption that $K'\cap K=\emptyset$. 
Hence $y_\infty\in K'$. 
Thus $\mathcal{G}\cup\{K'\}$ is also
upper semicontinuous.
\end{proof}

The following is our main separation theorem.
Recall that a compact set $K \subset \Sp^2$ \emph{separates}
$x$ and $y$ if they  belong to different components
of $\Sp^2 \setm K$.

\begin{thm}\label{thm-comp-separ-122}
 Let $T$ be a compact connected subset of\/ $\Sp^2$, 
and $\mathcal{K}:=\{K_j\}_{j}$ be a nonempty finite or countable
upper semicontinuous collection of pairwise disjoint compact  sets
such that\/ $\bigcup_{j} K_{j}$ is compact.

Let $x,y\in\Sp^2\setm(T\cup\bigcup_{j} K_j)$.
If for each $j$, the set $T\cup K_j$ does not separate $x$ from $y$,
then $T\cup\bigcup_{j} K_j$ does not separate $x$ and $y$ either.
\end{thm} 

In the special case when $T= \emptyset$ and all the $K_j$
are connected this follows
from Moore~\cite[Theorem~V.2.3, p.\ 312]{moore62}.

In order to prove this separation theorem we need to use
the following result.

\begin{thm} \label{thm-janiszewiski}
\textup{(Janiszewski's theorem)}
If $F_1,F_2 \subset \Sp^2$ are compact and such that
$F_1 \cap F_2 $ is connected, and
$x,y \in \Sp^2 \setm (F_1 \cup F_2)$ are not separated by
$F_1$ nor by $F_2$, then they are not separated by $F_1 \cup F_2$.
\end{thm}

This result was obtained by 
Janiszewski~\cite[Theorem (Twierdwenie)~A, p.\ 48]{Janiszewski}.
For a proof in English see, e.g., 
Moore~\cite[Theorem~IV.20, p.\ 173]{moore62},
Newman~\cite[Corollary to Theorem~8.1, p.~101]{newman} or
Pommerenke~\cite[Theorem~1.9]{PomUniv}.
 
We also use the following simple lemma.

\begin{lem}  \label{lem-separate-cap}
Assume that $F_0\supset F_1\supset\ldots$ are compact sets which separate
$x$ and $y$. Then so does $F:=\bigcap_{j=0}^\infty F_j$.
\end{lem}

\begin{proof}
Suppose not. 
Then (since connectedness and pathconnectedness coincide for open sets)
there is a curve $\ga$ in the open set
$\Sp^2\setminus F$ connecting $x$ and $y$. 
Because $\{F_j\}_{j=0}^\infty$ is a countable
decreasing chain of compact sets, and $\gamma$ 
is also compact and disjoint from 
$F$, it follows from Cantor's encapsulation theorem (which is sometimes
called the finite intersection property) that $\gamma$ lies in 
the complement of some $F_k$, that is, $F_{k}$ does not separate $x$ from $y$. 
This is a contradiction  and hence $F$ separates $x$ and~$y$.
\end{proof}

\begin{proof}[Proof of Theorem~\ref{thm-comp-separ-122}] 
The case when $\K$ is finite follows directly from iteration of 
Janiszewski's theorem, and we assume therefore that 
$\K=\{K_j\}_{j=0}^\infty$ is infinite in 
the rest of the proof.

We will prove this theorem by contradiction. 
So suppose that $x$ and $y$ are separated by $T\cup\bigcup_{j=0}^\infty K_j$.

Let $\mathcal{C}$ denote the collection of all (necessarily nonempty) compact
sets of the form
$\bigcup_{j\in A}K_j$ such that $x$ and $y$ are separated
by $T \cup \bigcup_{j\in A}K_j$.
We will construct a decreasing sequence in $\mathcal{C}$.
Let first $C_0=\bigcup_{j=0}^\infty K_j$ which belongs to $\mathcal{C}$ by assumption.
Proceed inductively for $j=0,1,\ldots$ as follows. 
If $K_j \subset C_{j}$ and there is some $C' \in \mathcal{C}$
such that $K_j \not \subset C' \subset C_{j}$, then let $C_{j+1}$ be any
such $C'$; otherwise, let $C_{j+1}=C_{j}$.

Let $I = \bigcap_{j=0}^\infty C_j$. 
That $I \ne \emptyset$ follows from 
Cantor's encapsulation theorem
(which is sometimes called 
the finite intersection property).
The pairwise disjointness of the sets $K_j$ 
yields that $I$ is of the form $\bigcup_{j\in \At}K_j$
for some (nonempty) $\At$.
Let $\Kt:=\{K_j : j \in \At\}$.
Lemma~\ref{lem-separate-cap} applied to $F_j=T\cup C_j$ shows that 
$I\in\mathcal{C}$.

\medskip
\emph{Claim}.
\emph{There exists some $K_m \in \Kt$ for which 
$\overline{\bigcup_{K \in \Kt \setm \{K_m\}} K}$ and
$K_m$ are disjoint\/ \textup{(}that is, $K_m$ is isolated from 
the union of all the other $K \in \Kt$\textup{)}.}
\medskip

As $T \cup K_m$ does not separate $x$ and $y$, $\Kt$ has at least two elements.
Since the $K_j$ are pairwise disjoint and compact, 
the claim is trivial if $\At$ is finite,
and hence we may assume that $\At$ is infinite.
In order to show the claim  we will argue by contradiction and show that
then $\Kt$ is an infinite 
compact perfect Hausdorff topological space, when equipped with
the topology below.
Since such spaces are uncountable
(see e.g.\ 
Hocking--Young~\cite[p.~88, Theorem~2-80]{HY}), the fact that $\Kt$ is
countable causes a contradiction, thus showing that the claim is true.

Assume therefore that the claim is false.
(In fact we will only use this assumption in Step~3 below.)
We equip $\Kt$ with a topology as follows.
A sequence $\{K'_{j}\}_{j=1}^\infty$ in $\Kt$ is said to \emph{converge} to 
$K'\in\Kt$ if 
\begin{equation}   \label{eq-deff-top}
\lim_{j\to\infty}\, \inf\biggl\lbrace
       \eps>0 : K'_{j}\subset \bigcup_{x\in K'}B(x,\eps)\biggr\rbrace=0.
\end{equation}
To obtain a topology on $\Kt$ we say that a set 
$\mathcal{F}\subset\Kt$ is
\emph{closed} if whenever 
$\{K'_j\}_{j=1}^\infty$ is a sequence in $\mathcal{F}$ 
converging to some $K'\in\Kt$,
then $K'\in\mathcal{F}$.

\medskip
\emph{Step}~1. \emph{$\Kt$ is Hausdorff.}
If $K'_{1}, K'_{2}\in\Kt$ are two distinct sets, then 
choose $0<\eps<\frac{1}{2}\text{dist}(K'_{1},K'_{2})$
(which is positive since $K'_1$ and $K'_2$ are disjoint and compact).
Let $\mathcal{O}_1$ be the collection of all
$K\in\Kt$ for which $K\subset \bigcup_{x\in K'_{1}}B(x,\eps)=:U$.
In order to show that $\mathcal{O}_1$ is open, 
let $\{K''_j\}_{j=1}^\infty$ be a sequence in $\Kt \setm \mathcal{O}_1$
and assume that $K''_j \to K'' \in \Kt$.
Then there are points $x_j \in K''_j \setm U$.
As $\bigcup_{K \in \Kt} K \setm U = I\setm U$ is compact, it contains a point $x$ 
such that a subsequence of  $\{x_{j}\}_{j=1}^\infty$ converges to $x$.
Since $x_j\in K''_j\to K''$ and $K''$ is compact, 
we conclude from \eqref{eq-deff-top} that
$x \in K''$ and therefore $K'' \notin \mathcal{O}_1$,
showing that the complement of $\mathcal{O}_1$ is closed and thus
that $\mathcal{O}_1$ is open.

Similarly, let $\mathcal{O}_2$ be the
collection of all $K\in\Kt$ for which 
$K\subset \bigcup_{x\in K'_{2}}B(x,\eps)$, which is also open
and disjoint from $\mathcal{O}_1$.
Moreover
$K'_{1}\in\mathcal{O}_1$ and $K'_{2}\in\mathcal{O}_2$.
We have thus shown that the topology on $\Kt$ is Hausdorff.

\medskip

\emph{Step}~2. \emph{$\Kt$ is  compact.}
Let $\{\G_\alp\}_{\alp}$ be a (possibly uncountable) cover of $\Kt$ by open sets.
For each $j \in \At$ 
there is some element $\G_{\alpha_j}$ in this cover such that $K_j \in\G_{\alpha_j}$.
Let $\G'_j=\bigcup_{k\in \At, \ k \le j} \G_{\alpha_k}$.

Assume  that $\Kt \setm \G'_j$ is nonempty for all $j$.
Then there is $K'_j \in \Kt \setm \G'_j$.
Let $x_j \in K'_j$ be arbitrary.
By the compactness of $I=\bigcup_{K \in \Kt} K$ there exists some 
$K' \in \Kt$ and $x \in K'$ such that a subsequence
$x_{j_k} \to x$ as $k \to \infty$.
Suppose that $\{K'_{j_k}\}_{k=1}^\infty$ does 
\emph{not} converge to $K'$ in the above topology,
i.e.
\[
 \tau :=\limsup_{k\to\infty}\inf\biggl\lbrace
       \eps>0 : K'_{j_k}\subset \bigcup_{x\in K'}
         B(x,\eps)\biggr\rbrace > 0.
\]
Then for each  $n$ we can find $j_{k_n}>n$ such that 
\[
K'_{j_{k_n}}\setminus\bigcup_{x\in K'}B\bigr(x,\tfrac{1}{2}\tau\bigr) 
\ne \emptyset.
\]
Let $y_n$ be a point in this set. 
Again by the  compactness of $I=\bigcup_{K \in \Kt} K$,
the sequence $\{y_n\}_{n=1}^\infty$ 
has a converging subsequence, but its limit cannot be in $K'$, 
violating the upper semicontinuity of $\Kt$.
(That $\Kt$ is upper semicontinuous follows from
Theorem~\ref{thm-comp-usc}\,\ref{it-comp-usc-b}.)
Hence $K'_{j_k}$ converges to $K'$. Note that $K'=K_l$ for some
$l\in\At$. By construction $K'\in\G'_l$, and so because $\G'_l$ is open, the tail
end of the sequence $K'_{j_k}$ must lie in $\G'_l$, violating the choice of
$K'_{j_k}$. 

Thus there is some $j$ such that $\Kt=\G'_j$, from which
it follows that $\{\G_k: k \le j \text{ and }  k \in \At \}$ 
is a finite subcover of $\{\G_\alp\}_{\alp}$
covering $\Kt$, and hence $\Kt$ is compact.

\medskip

\emph{Step}~3.  \emph{$\Kt$ is  perfect.}
Let $K' \in \Kt$.
Then by assumption
$K'$ is not isolated from $\bigcup_{K \in \Kt \setm \{K'\}} K$.
Thus there is a sequence $K'_j \in \Kt\setm\{K'\}$
and points $x_j \in K'_j$ and $x_0 \in K'$ such that $x_j \to x_0$
as $j \to \infty$.
It follows as in Step~2 that $K'_j \to K'$, showing that
$\Kt$ is perfect.

\medskip
The above three steps together indicate that $\Kt$ must be uncountable, which is not 
possible. This completes the proof of the claim.

Now let $K_m \in \Kt$ be as in the claim, and 
let $\Kt'=\Kt \setm \{K_m\}$.
It follows that $I':=\bigcup_{K \in \Kt'} K=I\setm K_m$ is a closed
subset of the compact set $I=\bigcup_{K \in \Kt} K$, 
and thus must be compact.
As $K_m \not\subset I' \subset C_{m}$ and $K_m\subset I$,
it follows from the construction of $C_{m+1}$
that
$I'\notin \mathcal{C}$,
i.e.\ $T \cup I'$ does not separate $x$ and $y$.
By assumption neither does $T \cup K_m$.
As $(T \cup K_m) \cap (T \cup I') = T$ is connected, it follows from
Janiszewski's theorem (Theorem~\ref{thm-janiszewiski})
that $x$ and $y$ are not separated by $T \cup I$.

This is our final contradiction and we have thus completed the proof.
\end{proof}

We now formulate and prove the main result of this section.

\begin{thm}   \label{thm-newman-count-singleton}
Let\/ $\Om\subset\R^2$ be bounded and assume 
that\/ $\Sp^2\setm\Om$ consists of\/ {\rm(}possibly uncountably many\/{\rm)} 
singleton components and at most countably many continuum components.
Let $E$ be the union of these continuum components and assume that 
$\itoverline{E}\setm E=F\cup E'$, where $F$ is compact and $E'$ is
at most countable.

If every boundary point $x_0\in\bdry\Om$ is locally accessible
from\/ $\Om$, then\/ $\Om$ is finitely connected at the boundary. 
\end{thm}

\begin{proof}
Fix $x_0 \in \bdy \Om$ and let $K$ be the component 
of $\Sp^2\setm\Om$ containing $x_0$. 
Then $\Om'=\Sp^2 \setm K$ is simply connected.
We shall show that $\Om'$ satisfies condition \ref{it-locacc}
of Theorem~\ref{thm-newman}, that is, each $x\in\partial\Om'$ is locally 
accessible from $\Om'$.

Let therefore $x\in\bdry\Om'\subset K$ and $r>0$ be arbitrary.
As $x$ is locally accessible from $\Om$, there exists $0<\de <r/2$ such that
every $y \in \Om \cap B(x,2\de)$ can be connected to $x$ by a curve in 
$\Om\cap B(x,r)$.
We want to show that $x$ is locally accessible also from $\Om'$.
Let $y \in (\Om'\setm\Om) \cap B(x,\de)$ and 
$z$ be a closest point to $y$ on $\bdy \Om$. 
Clearly, $z$  belongs to the same component of $\Sp^2\setm\Om$ as $y$.
In particular, $z \notin K$ and $|y-z| < \dist(y,K) \le |y-x| < \de$.
Since $z\in\bdry\Om$, there exists $y'\in \Om$ such that 
$|y'-z| < \dist(y,K) - |y-z| \le \dist(z,K)$.
It follows that  the line segment from $y$ to $y'$ does not hit $K$ 
and thus lies entirely in $\Om'$. As
$ |x-y'|\le |x-y|+|y-y'|<2\de $, 
there exists a curve in $\Om \cap B(x,r)\subset \Om'\cap B(x,r)$ 
connecting $y'$ and $x$.
Combining it with the line segment from $y$ to $y'$ provides a curve
in $\Om' \cap B(x,r)$ connecting $y$ to $x$.
This shows that $x$ is locally accessible from $\Om'$.
Since $\Om'$ is simply connected and satisfies condition~\ref{it-locacc}
of Theorem~\ref{thm-newman},
it follows from
Theorem~\ref{thm-newman} that $\Om'$ is finitely connected at the boundary,
and in particular at $x_0$. 
Our aim is to prove that also  $\Om$ is finitely connected at $x_0$.
 
If $K=\{x_0\}$ is a singleton, then let $r_0>0$ be arbitrary and $G:=B(x_0,r_0)$.
If $K$ is a continuum, then let $r_0>0$ be so small that 
$K\setm B(x_0,r_0)\ne\emptyset$ and let $B_0=B(x_0,r_0)$.
Then there are only finitely many components of $\Om'\cap B_0$ 
with $x_0$ in their closure, as $\Om'$ is finitely connected at $x_0$.
Let $G$ be one of these components. 
We will now show that $G$ is simply connected. Suppose not.
Then $\Sp^2\setm G$ has at least two components, each of 
which
is compact. 
Let $\widetilde{F}$ be the component of $\Sp^2\setm G$ 
containing the connected
set $K\cup(\Sp^2\setm B_0)$ (which is connected by Lemma~\ref{lem-connected}
and the choice of $r_0$).
Let $\widehat{F}$ be any other component of $\Sp^2\setm G$. Then clearly 
$\widehat{F}\subset \Om'\cap B_0$.
On the other hand, $\widehat{F}\cap\itoverline{G}$ is nonempty. 
Therefore by Lemma~\ref{lem-connected} again, $\widehat{F}\cup G$ is a 
connected subset of $\Om'\cap B_0$, which violates the fact that $G$ 
is a component of $\Om'\cap B_0$. Hence $G$ is simply connected.
(If $K$ is a singleton, then $G=B_0$ is clearly simply connected.)

We shall now prove that $\Om$ is finitely connected at $x_0$. 
Since $x_0$ is locally accessible from $\Om$, 
Theorem~\ref{thm-newman-gen} shows that 
$x_0 \notin \itoverline{\Hh_\Om (r_0,x_0)}$,
i.e.\ there exists $\de>0$ such that 
$B(x_0,\de)\cap \Hh_\Om (r_0,x_0)=\emptyset$.
Let $z,w\in G\cap\Om\cap B(x_0,\de)$. 
We will show that $z$ and $w$ cannot be separated in $G\cap\Om$.

Let $F_j$, $j=1,2\ldots$, denote all the 
(at most countably many)
continuum components of $\Sp^2\setm\Om$
(including $K$ if it is nonsingleton),  together with 
all the (at most countably many) singleton components of $E'\setm F$.
Note that $F\subset (\Sp^2\setm\Om)\setm E$ consists only of 
singleton components of 
$\Sp^2\setm\Om$ and is therefore totally disconnected.
Note also that if $K=\{x_0\}$ is singleton, we may have $x_0\in F\cup E'$.

Let $T_j=(F_j\cap G)\cup(\Sp^2\setm B_0)$, $j=1,2,\ldots$\,.
We shall show that $T_j$ does not separate $z$ and $w$.
This is clear if $F_j=K$.
If $F_j\ne K$, then $\dist(x_0,F_j)>0$ and hence also $\dist(x_0,T_j)>0$. 
As $z,w\notin \Hh_\Om (r_0,x_0)$, they belong to some components of
$\Om\cap B_0$ with $x_0$ in their closures.
Since 
$\Om\cap B_0  \subset B_0\setm F_j \subset \Sp^2\setm T_j$, 
we see that
those components of $\Om\cap B_0$ are subsets of $z$'s and $w$'s component(s)
in $\Sp^2\setm T_j$, which thus must contain $x_0$
(since $\dist(x_0,T_j)>0$).
But there can only be one such component in $\Sp^2\setm T_j$, so $T_j$ does
not separate $z$ and $w$.

Clearly, $T:=\Sp^2\setm G$  
does not separate $z$ and $w$.
Since $F$ is totally disconnected it 
has topological dimension zero, 
see e.g.\ Fedorchuk~\cite[Theorem~5 in Section~1.3.2]{fedorchuk}, 
and thus cannot separate the open set $G$, which
has topological dimension $2$, see 
e.g.\ \cite[Theorem~20 in Section~2.7]{fedorchuk}.
Thus $F \cup T$ 
does not separate $z$ and $w$.
Further, $T_j=(F_j\cap \itoverline{G})\cup(\Sp^2\setm B_0)$ is compact
and $T_j\cap T = \Sp^2\setm B_0$ is connected, so 
Janiszewski's theorem (Theorem~\ref{thm-janiszewiski}) implies that
$T_j\cup T = F_j\cup T$ does not separate $z$ and $w$ either.

As $\Sp^2\setm\Om$ 
and each of its components is compact, 
the collection $\{F_j\}_j$ is
upper semicontinuous, 
by
Theorem~\ref{thm-comp-usc}\,\ref{it-comp-usc-a} and~\ref{it-comp-usc-b}.
Theorem~\ref{thm-comp-usc}\,\ref{it-comp-usc-c} then yields that 
$\{F_j\}_{j}\cup\{F\}$ is also an upper semicontinuous collection
of pairwise disjoint compact sets.
By Theorem~\ref{thm-comp-separ-122} we now obtain that the set
\[
T\cup\itoverline{E}= T \cup \bigcup_{j} F_j \cup F
\]
does not separate $z$ and $w$.
Finally, $(\Sp^2\setm\Om)\setm\itoverline{E}$ is totally disconnected, so it
does not separate $\Sp^2\setm(T\cup\itoverline{E})$, i.e.\
$z$ and $w$ belong to the same component of 
\[
\Sp^2\setm((T\cup\itoverline{E})\cup((\Sp^2\setm\Om)\setm\itoverline{E}))
= \Sp^2\setm(T\cup(\Sp^2\setm\Om)) = G\cap\Om \subset \Om\cap B_0.
\]
This implies that each component $G$ of $\Om'\cap B_0$ with $x_0$ in its closure
contains exactly one component of $\Om\cap B_0$ with $x_0$ in its closure.
Indeed, if there were two such components $G_1$ and $G_2$ of $\Om\cap B_0$,
both contained in $G$, then we could find
$z\in G_1\cap B(x_0,\de)$ and $w\in G_2\cap B(x_0,\de)$, 
and the above argument
would give that $z$ and $w$ belong to the same component of $\Om\cap G$, 
which is a contradiction.

Since $\Om'$ is finitely connected at $x_0$ (or $G=B_0$ for singleton 
$K=\{x_0\}$), 
this implies that $\Om$ is also finitely connected.
As $x_0$ was arbitrary, this completes the proof. 
\end{proof}

\end{document}